\documentclass[12pt, a4paper]{article}
\usepackage{amssymb,amsmath,amsthm}
\usepackage{hyperref}
\usepackage{bm}
\usepackage[]{graphicx}

\usepackage{enumerate}
\usepackage{color}
\usepackage{verbatim}

\numberwithin{equation}{section}

\newcommand{\CL}{{\cal L}}
\newcommand{\CM}{{\cal M}}
\newcommand{\CN}{{\cal N}}
\newcommand{\CO}{{\cal O}}
\newcommand{\CQ}{{\cal Q}}
\newcommand{\CH}{{\cal H}}
\newcommand{\CA}{{\cal A}}

\newcommand{\R}{{\mathbb R}}
\newcommand{\N}{{\mathbb N}}
\newcommand{\E}{\mathbb{E}}

\newcommand{\prob}{\mathbb{P}}
\newcommand{\eps}{\varepsilon}
\newcommand{\norm}[1]{\left\lVert#1\right\rVert}
\newcommand{\trace}{\mathrm{trace}}

\newtheorem{theorem}{Theorem}[section]
\newtheorem{lemma}[theorem]{Lemma}
\newtheorem{definition}[theorem]{Definition}
\newtheorem{remark}[theorem]{Remark}
\newtheorem{proposition}[theorem]{Proposition}

\begin{document}
\title{Stochastic Cahn-Hilliard equation in higher space dimensions: The motion of bubbles}
\author{Alexander Schindler, Dirk Bl\"omker \\ Universit\"at Augsburg}
\date{\today}
\maketitle

\begin{abstract}
We study the stochastic motion of a droplet in a stochastic Cahn-Hilliard equation
in the sharp interface limit for sufficiently small noise.
The key ingredient in the proof is a deterministic slow manifold, where we show its stability for long times under 
small stochastic perturbations. 
We also give a rigorous stochastic differential equation
for the motion of the center of the droplet.
\end{abstract}

\section{Introduction}

In this work we consider the stochastic Cahn-Hilliard equation 
(also known as the Cahn-Hilliard-Cook equation \cite{Cook}) posed on 
a two-dimensional bounded smooth domain $\Omega \subset \mathbb{R}^2$: 
\begin{align}
\partial_t u &= - \Delta ( \eps^2 \Delta u - F^\prime (u) ) + \partial_t {W}(x,t), \, &x &\in \Omega, \nonumber \\
\partial_n u &= \partial_n \Delta u  = 0, \, &x &\in \partial \Omega.
\label{eq}
\end{align}
Here, $\eps$ is a small positive parameter measuring the relative 
importance of surface energy to the bulk free energy, 
and $\partial_n$ denotes the exterior normal derivative to the boundary $\partial \Omega$.
$F$ is assumed smooth with two equal nondegenerate minima, at $u = \pm 1$. 
A typical example is $F(u) = \frac14 (u^2 - 1)^2$.
We focus on this special case here, although most of the results hold for a very general class of 
nonlinearities. Only the precise formulation of the  stability result
and the condition on the noise strength there will change depending on the growth of $F$
at $\infty$.

The forcing is given by an additive white in time noise $\partial_t{W}$.
As we rely for simplicity of presentation on It\^o's formula, 
we assume that    the Wiener process is  sufficiently  smooth in space, 
and moreover sufficiently small in $\eps$, so that it does not destroy 
the typical patterns in the solutions.  

The existence and uniqueness of solutions is well-studied (see e.g.\cite{DPDe:96, CaWe:01})
and we always assume that we have a unique solution. Moreover, as we assume the noise to be smooth in space, 
the solution should be regular in space, too.

The deterministic Cahn-Hilliard equation is a gradient flow in the $H^{-1}$-topology 
for the following energy 
\[
 E_\eps(u) = \int_\Omega \frac{\eps^2}{2}|\nabla u|^2 +F(u)\; \mathrm{d} x 
\]
In order to minimize this energy,
one can expect that, for $0<\eps \ll 1$, solutions of (\ref{eq})
stay mostly near $u = -1$ and $u = +1$, the stable minima of $F(u)$.
Moreover, the gradient can be of order $\eps^{-1}$, so we expect small transition layers with thickness of order $\eps$.
Because of this, we can think of $\Omega$ as split into subdomains on
which $u_\eps (\cdot, t)$ takes approximately the constant values $-1$ and $1$, with boundaries
$\eps$-localized about an interface $\Gamma_\eps(t)$.

The interface is expected to move according to a Hele-Shaw or Mullins-Sekerka problem, 
where circular shaped droplets are stable stationary solutions of the dynamics. 
In \cite{DropletAC} formal derivation suggested a stochastic 
Hele-Shaw problem in the limit in case the noise is of order $\eps$.
There it was also shown that for very small noise, the dynamics is 
well approximated by a deterministic Hele-Shaw problem,
see also \cite{BaYaZh:19}. Also in \cite{YaZh:19} (or \cite{Ch:96} in the deterministic case)
the dynamics of the interface in the sharp interface limit was studied, 
but without obtaining an equation on the interface.
A rigorous discussion of the sharp interface limit in the deterministic case can be found in 
\cite{AlBaCh:94}.

In our result we focus on the almost final stage where the interface is already a single spherical droplet in the domain, 
and thus the only possible dynamics is given by the translation of the droplet, at least as long as the droplet stays away from the boundary. 
The deterministic case was studied in \cite{BubbleCH,CHBoundary}
and it was shown that the droplet moves (in $\eps$) exponentially slow.
Due to noise, we expect here a dominant stochastic motion of the droplet on a faster time-scale than exponentially slow.

As we want to study a single small droplet, the average mass of the solution is close to $\pm1$. 
In this regime an initially constant solution is locally stable, and one has to wait for a large deviation event,
that leads to the nucleation of droplets. See for example \cite{Ba:93,BaFi:93, DBGaWa:10,DBSaWa:16,DBMPWa:05}.

Let us finally remark, that although the result in \cite{BubbleCH,CHBoundary} holds also for three spatial dimensions, 
we focus here on the case of dimension $d=2$ only. 
With our method presented it is straightforward to 
treat the three dimensional case, only the technical details will change. 
Details will be provided in \cite{Schindler:Diss}.
Moreover, the case of the mass-conservative Allen-Cahn equation is similar.
See also \cite{DropletAC} for the motion of a droplet along the boundary, 
or \cite{AlChFu:00,BaJi:14} for the deterministic case.
For the one dimensional case see \cite{AnDBKa:12}.

\subsection{Assumptions on spaces and noise}

We fix the underlying space $H^{-1}(\Omega)$ with scalar product  
$\langle \cdot,\cdot\rangle$ and norm $\|\cdot\|$.

The standard scalar-product in $L^2(\Omega)$ 
is denoted by $(\cdot,\cdot)$ or $\langle \cdot,\cdot\rangle_{L^2}$.
Moreover, we use $\Vert \cdot \Vert_\infty$  for the supremum norm in $C^0$ or $L^\infty$.
 
As the Cahn-Hilliard equation preserves mass, we also consider the
subspace $H^{-1}_0(\Omega)$ of the Sobolev space $H^{-1}(\Omega)$ with zero average.
Recall, that the inner product in $H_0^{-1}(\Omega)$ is given by
\[
\langle \psi, \phi \rangle_{H^{-1}} = \left( (-\Delta)^{-1/2} \psi, (-\Delta)^{-1/2} \phi \right)_{L^2},
\]
where $-\Delta$ is the self-adjoint positive operator defined in 
$L^2_0(\Omega)= \{\phi \in L^2 (\Omega)| \int_\Omega \phi \, \mathrm{d}x = 0 \}$ 
by the negative Laplacian with Neumann boundary conditions.

Let $W$ be a $\CQ$-Wiener process in the underlying Hilbert space 
$H^{-1}(\Omega)$, where $\CQ$ is a symmetric operator 
and $(e_k)_{k \in \N}$ forms a complete 
$H^{-1}(\Omega)$-orthonormal basis of eigenfunctions of $\CQ$
with corresponding non-negative eigenvalues $\alpha_k^2$, i.e.
\[
\CQ e_k = \alpha_k^2 e_k.
\]
It is well known that $W$ is given as a Fourier series in $H^{-1}$
\begin{equation}
W(t) = \sum_{k=1}^\infty \alpha_k \beta_k(t) e_k,
\end{equation}
for a sequence of independent standard real-valued Brownian motions 
$\left\{ \beta_k(t) \right\}_{k \in \N}$, cf. DaPrato and Zabzcyck \cite{DaPrato}.

In order to guarantee mass-conservation of solutions to \eqref{eq},
the process $W$ is supposed to take values in $H^{-1}_0$ only, i.e. it satisfies
\begin{equation}
\int_\Omega W(t,x) \, \mathrm{d}x = 0 \quad \text{for any} \,\, t \geq 0\;.
\tag{N1}
\end{equation}

In order to simplify the presentation,  we rely on It\^o's formula.
Thus we have to assume that the trace of the operator $\CQ$ in $H^{-1}$ is finite, i.e.
\begin{equation}
\trace(\CQ) = \sum_{k=1}^\infty \alpha_k^2 =: \eta_0 < \infty.
\tag{N2}
\end{equation}
Furthermore, let $ \norm{\CQ}$ be the induced $H^{-1}$-operator norm of $\CQ$. 
\begin{equation}
\norm{ \CQ }_{L(H^{-1})}  =: \eta_1.
\tag{N3}
\end{equation}
Note that one always has
\[
  \eta_1 = \norm{ \CQ}_{L(H^{-1})}   \leq \trace (\CQ) = \eta_0.
\]
We assume that the Wiener process and thus $\CQ$ depends on $\eps$, 
and thus the noise strength is defined by either $\eta_0$ or $\eta_1$.

In the sequel for results in $L^2$-spaces, we need also higher regularity of $\CQ$.
For this purpose define the trace of $-\Delta Q$ in $H^{-1}$ by
\[
\trace(- \Delta \CQ) = \sum_{k=1}^\infty \langle - \Delta \CQ e_k, e_k \rangle_{H^{-1}}
= \sum_{k=1}^\infty \alpha_k^2 \norm{e_k}_{L^2}^2 =: \eta_2 < \infty.
\]
Note that $e_k$ was normalized in $H^{-1}$ and not in $L^2$.
\subsection{Outline and main result}
In our main results we rely on the existence of a deterministic slow manifold.
This was already studied in detail in \cite{BubbleCH} or \cite{CHBoundary}, where a deterministic
manifold of approximate solutions was constructed that consists of translations of a droplet state,
see section \ref{sec:SM} for details. Crucial points are the spectral properties of linearized operators 
that allows to show that the manifold is locally attracting.

In the deterministic case solutions  are attracted to an exponentially small
neighborhood of the manifold and follow the manifold until the droplet hits the boundary.
Moreover, the motion of the interface is given by an ordinary differential equation.
In the stochastic case this is quite different.

In section \ref{sec:SDE} we derive the motion along the manifold by projecting 
the dynamics of the stochastic Cahn-Hilliard equation to the manifold.
This is a rigorous description of the motion that involves no approximation.
We will see that sufficiently close to the manifold, the dynamics is in first approximation given
by the projection of the Wiener process onto the slow manifold, 
which is a stochastic equation for the motion of the center of the droplet.

In section \ref{sec:stab} we consider the stochastic stability of the slow manifold 
first in $H^{-1}$ and then in $L^2$. This heavily relies on the deterministic stability 
and on small noise, but as both the equation and the noise strength depends on $\eps$ 
we cannot use standard large deviation results. We use a technical Lemma from \cite{DropletAC}
in order to show that with overwhelmingly high probability one stays close
to the slow manifold for very long times.
Due to the stochastic forcing, we cannot exclude the possibility of rare events 
that will destroy the droplet or nucleate a second droplet.

Also the stability of the manifold holds for any polynomial time scale in $\eps^{-1}$, 
which is much larger than the time scale in which the droplet moves.
So we expect the droplet to hit the boundary at a specific polynomial time scale.

The final section \ref{sec:est} collects technical estimates used throughout the paper.

\section{The slow manifold}
\label{sec:SM}
Our stochastic motion of the droplet is based on the slow manifold constructed in \cite{BubbleCH} 
in the deterministic case.
In this section we collect some important results from \cite{BubbleCH} 
which we need throughout this work. We start with constructing the slow manifold 
$\tilde{\CM}^\rho_\eps$ consisting of translations of a single droplet with radius $\rho>0$
and discuss the spectrum of the linearized Cahn-Hilliard and Allen-Cahn operator afterwards.
These spectral properties are crucial in showing the stochastic stability
of the slow manifold. 

\subsection{Construction of the bubble}

We use  a bounded radially symmetric stationary solution to the Cahn-Hilliard equation 
on the whole space $\R^2$. As this solution (and all its derivatives) decays exponentially fast away from the droplet, its translations serve as 
good approximations for droplets inside the bounded domain.
A function $u \in C^2 (\R^2)$ is such a solution if, and only if, it is radial and satisfies 
\begin{equation}
\eps^2 \Delta u - F^\prime(u) = \sigma, \quad x \in \R^2
\end{equation}
for some constant $\sigma$. 
We also need some condition on monotonicity, in order to ensure that $u$ is a single droplet centered at the origin.

The following proposition, cf. \cite{BubbleCH} Thm. 2.1, concerns the existence of such radial solutions 
of the rescaled PDE
\begin{equation}
 \Delta u - F^\prime(u) = \sigma, \quad x \in \R^2.
\label{e:stat}
\end{equation}

\begin{proposition}
\label{thm:stat}
There exists a number $\bar{\rho} > 0$ and smooth functions $\sigma: (\bar{\rho}, \infty) \to \R$, $U^\star: [0,\infty) \times (\bar{\rho},\infty) \to \R$, such that for each $\rho \in (\bar{\rho},\infty)$, $\sigma(\rho)$ and $u(x,\rho) = U^\star ( \vert x \vert , \rho)$
satisfy equation \eqref{e:stat}. Moreover, $U^\star(r,\rho)$ is increasing in $r$ and
\begin{enumerate}[i)]
\item $\sigma(\rho) = C\rho^{-1} + \CO(\rho^{-2}), \quad \sigma^\prime (\rho) = C \rho^{-2} + \CO(\rho^{-3})$
\item $U^\star(\rho,\rho) = \CO(\rho^{-1})$
\item $ 1 + U^\star(0,\rho) = \CO(\rho^{-1})$
\item $\lim\limits_{ r \to \infty} U^\star (r,\rho) = \alpha(\rho)$, \\
 where $C > 0$ is a constant and $\alpha(\rho)$ denotes the root close to $1$ 
 of the equation $F^\prime(\alpha) + \sigma(\rho) = 0$. 
\item $\alpha(\rho) - U^\star(r, \rho) = \CO( e^{- \nu(\rho)(r - \rho)}), 
 \quad r > \rho, \nu(\rho) =( F''(\alpha(\rho)))^{\frac12}$ \\
 $U_r^\star(r,\rho) = \CO( e^{-\nu(\rho)(r-\rho)})$
\item Let $U$ be the unique solution of $U'' - F'(U) = 0, \lim_{s \to \infty} U(s) = \pm 1, U(0) = 0.$ Then there exists a constant $C>0$ such that
\[
U^{\star}(r, \rho) = U(r- \rho) + C \rho^{-1} V(r-\rho) + \CO(\rho^{-2}), \quad r-\rho \geq -C\rho,
\]
where $V$ is a bounded function such that
\[
\int_{-\infty}^\infty F'''(U(\eta)) U'(\eta)^2 V(\eta) \, d\eta = 0.
\]
\end{enumerate}
\end{proposition}

Here we used the usual $\mathcal{O}$-notation, 
that a term is $\CO(g(\rho))$ if there exists a constant 
such that the term is bounded by $Cg(\rho)$ for small $\rho>0$.

For a fixed radius $\rho>0$ of the droplets and  a fixed distance $\delta>0$ from the boundary of the domain,
Proposition \ref{thm:stat} assures that we can associate with each center 
$\xi \in \Omega_{\rho + \delta} = \{ \xi : d(\xi , \partial \Omega) > \rho + \delta \}$ 
a droplet, which is a function $u^\xi : \Omega \to \R$
with the following properties:

\begin{enumerate}[a)]
\item It is an almost stationary solution of the Cahn-Hilliard equation in the sense
that it fails to satisfy the equation, or the boundary conditions, by terms which
are of the order $\CO(e^{-c/ \eps})$ (including their derivatives)
\item It jumps from near $-1$ to near $1$ in a thin layer with thickness
of order $\eps$ around the circle of radius $\rho$ and center $\xi$.

\end{enumerate}

For $\eps \ll 1$ we define the droplet state
\begin{equation}
\label{defuxi}
u^\xi(x) = U^\star \left( \frac{\vert x - \xi \vert}{\eps} , \frac{\rho - a^\xi}{\eps}\right), \,\,\, x \in \Omega,
\end{equation}

where the number $a^\xi$ is chosen to be zero at some fixed $\xi_0 \in \Omega_{\rho + \delta}$ and is
determined for generic $\xi \in \Omega_{\rho + \delta}$ by imposing that the “mass” of $u^\xi$ is constant
on $ \Omega_{\rho + \delta}$, i.e.,
\begin{equation}
\label{masscorr}
 \int_\Omega u^\xi \, \mathrm{d}x = \int_\Omega u^{\xi_0}\,\mathrm{d}x , 
 \quad \forall \xi \in  \Omega_{\rho + \delta}\;.
\end{equation}

For example, we choose $\xi_0$ to be a point of maximal distance from the boundary $\partial\Omega$.
We could also fix a small mass and then determine the radius $\rho>0$ 
such that the droplet centered at $\xi_0$ has exactly that mass.

An easy argument based on Proposition \ref{thm:stat} v) shows (cf. Lemma 3.1 in \cite{BubbleCH})  that
\begin{equation}
\label{eq:mass}
a^\xi = \CO(e^{-c/\eps}),
\end{equation}
with similar estimates for the derivatives of $a^\xi$ with respect to $\xi_i$.

\subsection{The Quasi-Invariant Manifold and Equilibria}

In this section we state the construction of a manifold $\tilde{\CM}_\rho^\eps$ of droplets 
of the form $\xi \mapsto u^\xi + v^\xi$, where $v^\xi$ is a tiny perturbation, 
such that $\tilde{\CM}_\rho^\eps$ is an approximate invariant manifold for equation \eqref{eq}.
The construction of $\tilde{\CM}_\rho^\eps$ is made in such
a way that stationary solutions to \eqref{eq} with approximately circular interface are
in $\tilde{\CM}_\rho^\eps$
and can be detected by the vanishing of a vector field $\xi \mapsto c^\xi$.
Here we follow \cite{BubbleCH}.

\begin{theorem} 
\label{thm:slowmf}
Assume that $\rho > 0$ is such that $\Omega_\rho = \{\xi \in \Omega : d (\xi, \partial \Omega) > \rho \}$ is
non-empty and let $\delta > 0$ be a fixed small number. Then there is an $\eps_0 > 0$ such
that, for any $0 < \eps < \eps_0$ there exist $C^1$ functions
\begin{equation}
\xi \mapsto v^\xi \in C^4( \bar{\Omega} ), \,\, \xi \mapsto c^\xi = (c_1^\xi , c_2^\xi) \in \R^2 
\end{equation}
defined in $\Omega_{\rho + \delta}$ and such that $\int_\Omega v^\xi \, \mathrm{d}x = 0$, for which

\begin{enumerate}[(i)]
\item $\norm{v^\xi}_\infty \leq C \eps^{-2} e^{-(v_\eps/\eps)d^\xi}$,
\item $\vert c^\xi \vert \leq C \eps^{-4} e^{-2(v_\eps/\eps)d^\xi}$
\item Similar estimates with $C$ replaced by $C \eps^ {-k}$ , with $k$ the order of differentiation, hold for the derivatives of $v^\xi$, $c^\xi$ with respect to $x$, $\xi$.
\item The function $\tilde{u}^\xi = u^\xi + v^\xi$ satisfies the boundary conditions in (\ref{eq}) and
\begin{equation*}
\CL(\tilde{u}^\xi) = c_1^\xi u_1^\xi + c_2^\xi u_2^\xi,
\end{equation*}
where $\CL(\Phi) = \Delta ( - \eps^2 \Delta \Phi + F^\prime (\Phi) )$ and $u_i^\xi$ is the derivative of $u^\xi$ with respect to $\xi_i$, $i = 1, 2$.
\item Let $\tilde{\CM}_\rho^\eps \subset C^0(\bar{\Omega})$ be the two-dimensional manifold
\begin{equation*}
\tilde{\CM}_\rho^\eps = \left\{ u = \tilde{u}^\xi : \xi \in \Omega_{\rho + \delta} \right\} 
\end{equation*}
and let $\tilde{\CN}_\eta \subset C^0(\bar{\Omega})$ be the open neighbourhood of $\tilde{\CM}_\rho^\eps$ defined by
\begin{equation*}
\tilde{\CN}_\eta = \{ u : \exists \xi \in \Omega_{\rho + \delta}, w \in C^0 ( \bar{\Omega}), \norm{w}_\infty < C \eps^\eta, u = \tilde{u}^\xi + w \}.
\end{equation*}
 Then there is a sufficiently small 
$\eta > 0$ such that $u \in \tilde{\CN}_\eta$ is an
equilibrium of (\ref{eq}) if and only if
\begin{equation*}
u = \tilde{u}^\xi, \,\,\, c^\xi = 0
\end{equation*}
for some $\xi \in \Omega_{\rho + \delta}$.
\end{enumerate}
\end{theorem}


\subsection{Spectral estimates for the linearized  operators}

An essential point in the stochastic stability are the spectral properties of 
the linearized Cahn-Hilliard and Allen-Cahn operator.
We consider the linearization around 
any droplet state in our slow manifold,
and it is crucial that eigenfunctions not tangential to the manifold have negative eigenvalues uniformly bounded away from zero, 
while all other eigenvalues have eigenfunctions tangential to the manifold.

\subsubsection{The Cahn-Hilliard operator on \boldmath$H^{-1}_0(\Omega)$}

We study the linearized Cahn-Hilliard operator
\[
\CL^\xi  = \Delta \left( - \eps^2 \Delta + F''( \tilde{u}^\xi ) \right)
\]
in more detail. 
We consider $\CL^\xi$ as an operator on $H_0^{-1}(\Omega)$ and cite a theorem of \cite{SpecEst} below.

As we have exponentially small terms, we use the following definition:
\begin{definition}
We say that a term is of order $\CO(\exp)$ if it is asymptotically exponentially small as $\eps \to 0$,
i.e. of order $\CO( e^{-c/\eps} )$ for some positve constant $c$.
\end{definition}

\begin{theorem}
\label{LinCH}\ 
\begin{enumerate}[(i)]
\item
The operator $\CL^\xi$ can be extended to a self-adjoint operator on $H_0^{-1}$ , the subspace of the Sobolev space $H^{-1}$ consisting of functions with zero average. $−\CL^\xi$ is bounded from below.
\item
Let $\lambda^\xi_1 \leq \lambda^\xi_2 \leq \lambda^\xi_3 \leq \ldots$  be the eigenvalues of
\begin{align*}
\CL^\xi \psi &= \Delta \left( - \eps^2 \Delta \psi + F''(\tilde{u}^\xi) \psi \right) = - \lambda \psi, \quad &x \in \Omega, \\
\frac{\partial \psi}{\partial \eta} &= \frac{\partial \Delta \psi}{\partial \eta} = 0, \quad &x \in \partial \Omega,
\end{align*}
and let $\delta > 0$ be fixed. Then there is $\eps_0 > 0$ and constants $c, C , C' > 0$ independent of $\eps$ such that, for $0 < \eps < \eps_0$ and $\xi \in \Omega_\delta$, the following estimates hold:
\begin{align}
-C e^{-c/\eps} \leq &\lambda_1^\xi \leq \lambda_2^\xi \leq C e^{-c/\eps}, \\
&\lambda_3^\xi \geq C' \eps.
\end{align}
\item
In the two-dimensional subspace $U^\xi$ corresponding to the small eigenvalues
$\lambda_1^\xi , \lambda_2^\xi$ there is an orthonormal basis (in $H^{-1}$) $\psi^\xi_1 , \psi^\xi_2$ such that
\begin{equation}
\label{def:psi}
\psi_i^\xi = \sum_{j=1}^2 a_{ij}^\xi \frac{\tilde{u}_j^\xi}{\| \tilde{u}_j^\xi \| } + \CO(\exp), \quad i =1,2,
\end{equation}
where the matrix $(a_{ij}^\xi)$ is nonsingular and a smooth function of $\xi$ and $\tilde{u}^\xi_j$ is the
derivative of $\tilde{u}^\xi$ with respect to $\xi_j$ . Moreover $\psi_i^\xi$ is a smooth function of $\xi$ and
\begin{equation}
 \| \psi_{i,j} \| = \CO( \eps^{-1} ), \quad i,j = 1,2,
\end{equation}
where $\psi_{i,j}$ is the derivative of $\psi_{i}$ with respect to $\xi_j$.
\end{enumerate}
\end{theorem}
As we will need the statement in more detail later,
we will comment on the proof of (iii). 
The main ingredient is the following theorem. For its proof we refer to \cite{H-Sj}. 
\begin{theorem} 
\label{H-Sj}

Let $A$ be a selfadjoint operator on a Hilbert space $\CH$, $I$ a compact interval in $\R$, $\{ \psi_1, \ldots, \psi_N \}$ linearly independent normalized elements in $\mathcal{D}(A)$. We assume
\begin{enumerate}[(i)]
\item
\begin{align*}
\left\{
\begin{array}{l l}
A \psi_j &= \mu_j \psi_j + r_j, \quad \| r_j \| < \eps' \\
&\mu_j \in I, \quad j= 1,\ldots, N
\end{array}
\right\}
\end{align*}
\item
There is a number $a > 0$ such that $I$ is $a$-isolated in the spectrum of $A$:
\[
 \left( \sigma(A) \setminus I \right) \cap \left( I + (-a,a) \right) = \emptyset.
\]
\end{enumerate}
Then
\[
\bar{d}(E,F) := \sup_{ \phi \in E, \| \phi \| = 1} d(\phi,F) \leq \frac{\sqrt{N} \eps'}{a \sqrt{\lambda_{min}}},
\]
where
\begin{align*}
&E = \text{span} \{ \psi_1, \ldots , \psi_N \}, \\
&F = \text{closed subspace associated to } \sigma(A) \cap I, \\
&\lambda_{min} = \text{smallest eigenvalue of the matrix } \left( \langle \psi_i, \psi_j \rangle \right)_{i,j=1,\ldots,N}.
\end{align*}
\end{theorem}
In our case we take $E = \text{span}\{ \frac{\tilde{u}^\xi_1}{\| \tilde{u}^\xi_1 \|}, \frac{\tilde{u}^\xi_2}{\| \tilde{u}^\xi_2 \|} \}$, 
$I = \left[ - C e^{-c/\eps}, C e^{-c/ \eps} \right]$, such that $\sigma(A) \cap I = \{ \lambda^\xi_1, \lambda^\xi_2 \}$, and $a = \eps^2$. According to theorem \ref{LinCH}(ii) the spectral gap is of order $\eps$ and therefore $I$ is $a$-isolated.

Let us now discuss that the eigenvectors corresponding to the smallest eigenvalues approximate well 
the tangent space of the slow manifold.
First, the droplet state is an approximate solution, 
so for its derivative $ \tilde{u}^\xi_j $ (which is a tangent vector) we have
\[ 
\CL^\xi \frac{\tilde{u}^\xi_j}{\| \tilde{u}^\xi_j \|} = r_j \quad \text{with } \| r_j \| = \CO(\exp).
\]
Since the matrix $\left( \langle \tilde{u}^\xi_i, \tilde{u}^\xi_j \rangle \right)$ approaches a nonsingular limit as $\eps \to 0$
(see e.g. \eqref{e:Muxi}), 
we also have $ |\lambda_{min}| > C > 0$. For $i \in \{1,2\}$ we denote the associated eigenvector to $\lambda^\xi_i$ by $\psi^\xi_i$ and define $F = \text{span} \{\psi^\xi_1,\psi^\xi_2 \}$ . Theorem \ref{H-Sj} is applicable and yields
\[
\bar{d}(E,F) := \sup_{ \phi \in E, \| \phi \| = 1} d(\phi,F) = \CO( \exp ).
\]
Thus, $\psi^\xi_i \in E + \CO(\exp)$ and one can write
\begin{equation}
\label{e:defa}
\psi_i^\xi = \sum_{j=1}^2 a_{ij}^\xi \frac{\tilde{u}_j^\xi}{\| \tilde{u}_j^\xi \| } + \CO(\exp), \quad i =1,2.
\end{equation}

By definition of the distance $\bar{d}$ we have
\begin{equation*}
\frac{\tilde{u}_j^\xi}{\| \tilde{u}_j^\xi \| } 
 = \sum_k \langle \frac{\tilde{u}_j^\xi}{\| \tilde{u}_j^\xi \| }, \psi^\xi_k \rangle \, \psi^\xi_k + \CO(\exp).
\end{equation*}
Noting $\| \tilde{u}_j^\xi \| \leq C \rho$ we get by multiplying
\begin{equation*}
\tilde{u}_j^\xi 
 = \sum_k \langle \tilde{u}_j^\xi, \psi^\xi_k \rangle \, \psi^\xi_k + \CO(\exp).
\end{equation*}
It remains to show that the matrix $B(\xi)$ defined by $B_{jk}(\xi) = \langle \tilde{u}_j^\xi, \psi^\xi_k \rangle$ 
is invertible. This can be seen as follows:
\begin{align*}
\langle \tilde{u}_i^\xi, \tilde{u}_j^\xi \rangle 
&= \Big\langle \sum_k B_{ik}(\xi) \psi^\xi_k , \sum_l B_{jl}(\xi) \psi^\xi_l \Big\rangle + \CO(\exp) \\
&= \sum_{k,l} B_{ik}(\xi) B_{jl}(\xi) \langle \psi^\xi_k , \psi^\xi_l \rangle + \CO(\exp) \\
&= \sum_k B_{ik}(\xi) B_{jk}(\xi) + \CO(\exp) = \left( B \cdot B^T \right)_{ij} + \CO(\exp)
\end{align*}
Therefore invertibility of $B$ is equivalent to the invertibility of the matrix defined by $\langle \tilde{u}_i^\xi, \tilde{u}_j^\xi \rangle$ which is already proven. 

\begin{remark}
Note that Theorem \ref{LinCH} is restricted to the two-dimensional case. While the construction of an orthonormal basis as in (iii) is the same, thus far, for 
$d=3$ it can be shown that the spectral gap is  only  of order $\CO(\eps^2)$. This heavily influences our analysis of stochastic stability and any improvement
of this result will yield a better region of stability in the three-dimensional setting.
\end{remark}

\subsubsection{The mass-conserving Allen-Cahn operator on \boldmath$L^2_0(\Omega)$.}

Next, we collect some results on the eigenvalue problem for the mass-conserving Allen-Cahn equation 
linearized around $\tilde{u}^\xi$, for small  $0<\eps \ll 1$,

\begin{align}
&\CA^\xi \phi = \eps^2 \Delta \phi - F''(\tilde{u}^\xi) \phi - \frac{1}{\vert \Omega \vert}
\int_\Omega F''(\tilde{u}^\xi)\, \phi \, \mathrm{d}x = - \mu \phi, \quad &x \in \Omega, \nonumber \\
&\frac{\partial \phi}{\partial \eta} = 0, \quad &x \in \partial \Omega
\label{massAC}
\end{align}

on $L^2(\Omega)$. Here as defined previously $\tilde{u}^\xi$ is the bubble state, 
which is an element of the slow manifold. 

\begin{theorem}
\label{LinAC}
Let $\tilde{u}^\xi \in \tilde{\CM}_\rho^\eps$ and let $\mu_1 \leq \mu_2 \leq \mu_3 \leq \ldots$ be the eigenvalues of \eqref{massAC}. Then there is
$\eps_0$ such that for $\eps < \eps_0$
\begin{align}
\mu_1, \mu_2 = \CO(\exp) \\
\mu_3 > C\eps^2.
\end{align}
The two-dimensional space $W^\xi$ spanned by the eigenfunctions corresponding to the 
eigenvalues $\mu_1, \mu_2$ can be represented by 
\[
W^\xi = \mathrm{span} \left\{ h_1^\xi, h_2^\xi \right \}
\]
and 
\begin{equation}
\norm{ h_i^\xi - \frac{u_i^\xi}{\norm{u_i^\xi}}}_{L^2} = \CO(\exp).
\label{ACeigenfunction}
\end{equation}
\end{theorem}

This result can be found in \cite{AlBrFu:98} with $\tilde{u}^\xi$ replaced by $u^\xi$. As $v = \tilde{u}^\xi - u^\xi$ 
is exponentially small the theorem follows from an 
easy perturbation argument. 
Also note that for the eigenfunctions of Cahn-Hilliard we thus have by \eqref{ACeigenfunction} and \eqref{e:defa}
\[
\psi_i^\xi = \sum_j \alpha_{ij}^\xi h_j^\xi + \CO(\exp).
\]

\begin{remark}
\label{rem:MCAC}
Defining the projection $Pu = u - \frac{1}{\vert \Omega \vert}\int_\Omega u \, dx$ onto $L^2_0(\Omega) = \left\{ f \in L^2(\Omega) : \int_\Omega f = 0 \right\}$
we see that for $v \in H^{-1}_0$
\begin{align*}
\langle \CL^\xi v,v \rangle_{H^{-1}} &= \langle \eps^2 \Delta v - F''(\tilde{u}^\xi) v, Pv \rangle_{L^2} \\
&=  \langle P \circ (\eps^2 \Delta  - F''(\tilde{u}^\xi)) v, v \rangle_{L^2} \\
&= \langle \CA^\xi v,v \rangle_{L^2}\;.
\end{align*}
Therefore, for all $v \perp_{H^{-1}} \psi_i^\xi$ we have 
\[
\langle \CL^\xi v,v \rangle_{H^{-1}} \leq -C \eps^2 \norm{v}_{L^2}^2,
\]
which is crucial for establishing stability.
\end{remark}

\section{Motion along the slow manifold : The dynamics of bubbles}
\label{sec:SDE}

Here we follow the approach to split  
the dynamics into the motion along the manifold and othogonal to it.

\subsection{The new coordinate system}

We will use the standard projection onto the manifold. A minor technical difficulty 
is that the eigenfunctions $\psi^\xi_1$ and $\psi^\xi_2$ of the linearization 
do not span the tangent space at a given point $\tilde{u}^\xi$ on the slow manifold.
But as the difference to the true tangent space, which is spanned by the partial derivatives  $\partial_{\xi_1 }\tilde{u}^\xi$ and $\partial_{\xi_2 }\tilde{u}^\xi$,
is exponentially small, we can use them as an approximate tangent space to project onto the manifold.

The following proposition concerns the existence of a small tubular neighborhood of $\tilde{\CM}_\rho^\eps$ where the projection is well-defined,
see \cite{BubbleCH}.

\begin{proposition} \label{thm:Fermi}
Let $\tilde{u}^\xi$, $\tilde{\CM}_\rho^\eps$, $\Omega_\rho$ be as in Theorem \ref{thm:slowmf}; then, for $\eta > 1$, the
condition
\begin{equation}
\inf_{\xi \in \Omega_{\rho + 2\delta}} \Vert u - \tilde{u}^\xi \Vert 
< \eps^\eta,
\end{equation}
implies the existence of a unique pair $\xi \in \Omega_{\rho + \delta}$ , $v \in H_0^{-1}$ such that
\begin{align} \label{eq:Fermi}
u &= \tilde{u}^\xi + v \nonumber  \\
\langle v , \psi_i^\xi \rangle &= 0, \,\, i =1,2,
\end{align}
where $\psi_1^\xi, \psi_2^\xi$ form a basis of the two-dimensional subspace 
corresponding to the two smallest eigenvalues of the linearized operator $\CL^\xi$ 
and are given by theorem \ref{LinCH} (iii).
Moreover, the map $u \to (\xi, v)$ defined by \eqref{eq:Fermi} is a
smooth map together with its inverse.

\end{proposition}

Let $u(t)$ be a solution of \eqref{eq}. We will call the coordinates $v$ and $\xi$ defined in proposition \ref{thm:Fermi} the Fermi coordinates of $u(t)$.

\subsection{The exact stochastic equation for the droplet}

In the remainder of this section we adopt the approach of \cite{AnDBKa:12} 
and assume that the center $\xi$ of the bubble 
$\tilde{u}^\xi$ defines a multidimensional diffusion process which is given by
\begin{equation}
\label{e:ansatz}
d\xi_k = f_k (\xi) \, dt + \langle \sigma_k(\xi) , dW \rangle, 
\end{equation}
for some given vector field $f : \R^2 \to \R^2$ and some variance 
$\sigma : \R^2 \to \CH^2$.
We proceed with deriving explicit formulas for $f$ and $\sigma$, 
which still depend on the distance $v$ to the manifold.

\vspace{0.5 cm}

We use the It\^o formula, in order to differentiate (\ref{eq:Fermi}) with respect to $t$, and get
\begin{equation}
du=dv+\sum_j \tilde{u}_j^{\xi}d\xi_j +\tfrac12\sum_{i,j} \tilde{u}_{ij}^{\xi} d\xi_j  d\xi_i.
\end{equation}

Taking the inner product in the Hilbert space $H^{-1}$ with $\psi^\xi_k$ yields for any $k$
\begin{equation}
 \label{e:e1}
\langle \psi^\xi_k, du\rangle 
=\langle \psi^\xi_k, dv\rangle 
+\sum_j \langle \psi^\xi_k, \tilde{u}_j^{\xi}\rangle d\xi_j 
+\tfrac12\sum_{i,j} \langle \psi^\xi_k, \tilde{u}_{ij}^{\xi} \rangle  d\xi_j  d\xi_i.
\end{equation}
On the other hand taking the scalar-product of (\ref{eq}) with $\psi^\xi_k$  we derive 
\begin{equation}
  \label{e:e2}
\langle \psi^\xi_k, du\rangle 
=\langle \psi^\xi_k, \CL(v+\tilde{u}^\xi)\rangle dt
+ \langle \psi^\xi_k,  dW \rangle.
\end{equation}
Now (\ref{e:e1}) and (\ref{e:e2}) together imply 
\begin{eqnarray}
 \label{e:1st}
\sum_j \langle \psi^\xi_k, \tilde{u}_j^{\xi}\rangle d\xi_j 
&=&- \langle \psi^\xi_k, dv\rangle 
- \tfrac12\sum_{i,j} \langle \psi^\xi_k, \tilde{u}_{ij}^{\xi} \rangle  \langle \CQ \sigma^\xi_j,\sigma^\xi_i \rangle dt\\
&&+\langle \psi^\xi_k, \CL(v+\tilde{u}^\xi)\rangle dt 
+ \langle \psi^\xi_k, dW \rangle, \nonumber
\end{eqnarray}
where we also used that $\langle w, dW \rangle\langle g, dW \rangle
=\langle \CQ  w, g \rangle dt$.

In order to eliminate $dv$, we apply the It\^o formula to the orthogonality
condition $\langle \psi_k^\xi, v \rangle=0$ and arrive at 
\begin{eqnarray*}
 \langle dv , \psi^\xi_k \rangle &=&  -\langle v , d\psi^\xi_k \rangle
  -\langle dv , d\psi^\xi_k \rangle\\
&=& -\sum_j \langle v , \psi_{jk}^{\xi} \rangle d\xi_j
  -\tfrac12\sum_{i,j} \langle v , \psi_{ijk}^{\xi} \rangle d\xi_i d\xi_j
  -\sum_j \langle dv, \psi_{jk}^{\xi} \rangle  d\xi_j \\
&=& -\sum_j \langle v , \psi_{jk}^{\xi} \rangle d\xi_j
  -\tfrac12  \sum_{i,j} \langle v , \psi_{ijk}^{\xi} \rangle \langle \CQ \sigma_i^{\xi} , \sigma_j^{\xi} \rangle dt
  -\sum_j \langle dv, \psi_{jk}^{\xi} \rangle  d\xi_j \;.
\end{eqnarray*}
Now we use that $dv = du - d{\tilde{u}}^{\xi}$ and the fact that $dt dt = 0$ and $dW dt = 0$ and get
\begin{eqnarray*}
\lefteqn{ -\sum_j \langle dv, \psi_{jk}^{\xi} \rangle  d\xi_j}\\
&=& -\sum_j \langle du, \psi_{jk}^{\xi} \rangle  d\xi_j
 +\sum_j \langle d\tilde{u}^{\xi}, \psi_{jk}^{\xi} \rangle  d\xi_j \\
&=& -\sum_j \langle \CL(u) , \psi_{jk}^{\xi} \rangle dt d\xi_j
 - \sum_j \langle \psi_{jk}^{\xi} , dW \rangle d\xi_j
 +\sum_{i,j} \langle \psi_{jk}^{\xi} , \tilde{u}_i^{\xi} \rangle d\xi_i d\xi_j \\
&=& - \sum_j \langle \CQ \psi_{jk}^{\xi} , \sigma_j^{\xi} \rangle dt
 +  \sum_{i,j} \langle \psi_{jk}^{\xi} , \tilde{u}_i^{\xi} \rangle \langle \CQ\sigma_i^{\xi} , \sigma_j^{\xi} \rangle dt \;.
\end{eqnarray*}

This yields together with (\ref{e:1st})
\begin{eqnarray}
 \label{e:final}
\lefteqn{\sum_j \left[ \langle \psi^\xi_k, \tilde{u}_j^{\xi}\rangle - \langle v,  \psi_{jk}^{\xi}\rangle \right] d\xi_j}\nonumber \\
& = & 
\sum_{i,j} \left[ \tfrac12  \langle v, \psi_{ijk}^{\xi} \rangle - \langle \psi^\xi_{jk} , \tilde{u}^\xi_i \rangle - \tfrac12  \langle \psi^\xi_k, \tilde{u}_{ij}^{\xi} \rangle \right] \langle \CQ\sigma^\xi_i,\sigma^\xi_j\rangle dt
+  \sum_j \langle \CQ \psi^\xi_{jk},  \sigma_j^\xi\rangle dt \nonumber\\
&& +\langle \psi^\xi_k, \CL(v+\tilde{u}^\xi)\rangle dt 
+  \langle \psi^\xi_k, dW \rangle
\end{eqnarray}

Define the matrix $(A_{kj}(\xi))_{k,j}=A(v,\xi)\in\R^{2\times 2}$ by
\begin{equation}
 \label{def:A}
A_{kj}(\xi)= Z^0_{kj} + Z^1_{kj}(v)=\langle \psi^\xi_k, \tilde{u}_j^{\xi}\rangle - \langle v,  \psi_{k,j}^{\xi}\rangle.
\end{equation}

By theorem \ref{LinCH} (iii) we have $\Vert \psi^\xi_{i,j} \Vert = \CO( \eps^ {-1})$.
Therefore,as long as
\begin{equation*}
\inf_{\xi \in \Omega_{\rho + 2\delta}} \Vert u - \tilde{u}^\xi \Vert
= \inf_{\xi \in \Omega_{\rho + 2\delta}} \Vert v^\xi \Vert  < \eps^\eta \,\,\, \text{ for some } \eta > 1 
\end{equation*}
we have
\begin{equation}
\vert \langle \psi_{j,k}, v \rangle \vert \leq \Vert v \Vert \Vert \psi_{j,k} \Vert < \eps^ {\eta - 1}\;.
\end{equation}

In the comment to the proof of \ref{LinCH} we have seen that the matrix 
$\langle \psi^\xi_k, \tilde{u}_j^{\xi}\rangle$ 
is nonsingular and approaches a constant 
as $\eps \to 0$.
As a consequence we observe that the matrix $A(\xi)$ is invertible in a tube $\Gamma$ 
around $\tilde{\CM}_\rho^\eps$. This proof is straightforward. The  details are similar to Lemma \ref{lem:matrix} 
and the tube $\Gamma$ has radius $\eps^\eta$ for any fixed $\eta>1$.

We denote the entries of the inverse matrix by $A^{-1}_{kj}(\xi)$.
\vspace*{0.3cm}

From (\ref{e:final})
we derive
\begin{equation*}
\sum_j{A_{kj}(\xi)} \, \sigma_j^\xi = \psi^\xi_k
\end{equation*}
and 
\begin{eqnarray*}
 \sum_j{A_{kj}(\xi)}f_j(\xi)&=&
 \sum_{i,j} \left[ \tfrac12\langle v, \psi_{i,jk}^{\xi} \rangle - \langle \psi^\xi_{j,k} , \tilde{u}^\xi_i \rangle -\tfrac12 \langle \psi^\xi_k, \tilde{u}_{ij}^{\xi} \rangle \right] \langle \CQ\sigma^\xi_i,\sigma^\xi_j\rangle \nonumber\\
&& +  \sum_j \langle \CQ \psi^\xi_{j,k},  \sigma_j^\xi\rangle +\langle \psi^\xi_k, \CL(v+\tilde{u}^\xi)\rangle.
\end{eqnarray*}
Using the invertibility of $A(\xi)$ we finally get formulas for $f$ and $\sigma$:
\begin{equation}
\label{def:sigma}
\sigma_r(\xi) = \sum_i A_{ri}^{-1}(\xi) \psi^\xi_i
\end{equation}
and
\begin{align}
\label{def:f}
f_r(\xi) = &\sum_i A_{ri}^{-1}(\xi) \langle \CL(v+\tilde{u}^\xi), \psi^\xi_i\rangle \nonumber \\
&+ \sum_{i,j,k} A_{ri}^{-1}(\xi) \left[ \tfrac12\langle v, \psi_{i,jk}^{\xi} \rangle - \langle \psi^\xi_{j,k} , \tilde{u}^\xi_i \rangle -\tfrac12 \langle \psi^\xi_k, \tilde{u}_{ij}^{\xi} \rangle \right] \langle \CQ\sigma^\xi_i,\sigma^\xi_j \rangle \nonumber \\
&+ \sum_i A_{ri}^{-1}(\xi) \sum_j \langle \CQ \psi^\xi_{i,j},  \sigma_j^\xi\rangle.
\end{align}

\subsection{Verification of the SDE}

In the derivation, we made the assumption that $\xi$ is a semimartingale with respect to the Wiener process $W$.
We now prove that this assumption is indeed true. 
At least we find one splitting $u = \tilde{u}^\xi + v$, where $\xi$ is a semimartingale given by our derived SDE for $\xi$.

\begin{lemma}
Consider the pair of functions $(\xi,v)$ as solutions of the system given by
(\ref{eq:dv}) and the ansatz (\ref{e:ansatz}), where $\sigma$ and $f$ are given by (\ref{def:sigma}) and (\ref{def:f}). 
Suppose that initially $\langle \psi^{\xi(0)}_k,v(0) \rangle = 0$ 
for $k = 1,2$.

Then $u = \tilde{u}^\xi + v$ solves (\ref{eq}) with $\langle \psi^\xi_k,v \rangle = 0$
for $k = 1,2$.
\end{lemma}

\begin{proof}
We first prove that $u = \tilde{u}^\xi + v$ solves (\ref{eq}).
\begin{eqnarray*}
 du
 &=& \tilde{u}^\xi + dv \\
 &=& \sum_j \tilde{u}^\xi_j \, d\xi_j + \tfrac12\sum_{i,j} \tilde{u}_{ij}^\xi \, d\xi_i \, d\xi_j +dv \\
 &=& \sum_j \tilde{u}^\xi_j \, d\xi_j + \tfrac12\sum_{i,j} \tilde{u}_{ij}^\xi \, d\xi_i \, d\xi_j +
 \CL(v+\tilde{u}^\xi) \, dt +  \, dW \\
 && -\sum_j  \tilde{u}_j^{\xi} \, d\xi_j -\tfrac12 \sum_{i,j}  \tilde{u}_{ij}^{\xi}  \langle  Q\sigma^\xi_j,\sigma^\xi_i\rangle \, dt \\
 &=& \CL(v+\tilde{u}^\xi)\,dt + \, dW \\
 &=& \CL(u)\,dt + \, dW
\end{eqnarray*}
The orthogonality condition follows from $d\langle  v,\psi_k^\xi \rangle = 0$ since $v(0) \perp T_{\tilde{u}^{\xi(0)}} \CM$. 
We have
 \begin{eqnarray*}
d\langle  v,\psi_k^\xi  \rangle &=& \langle dv , \psi_k^\xi \rangle + \langle v, d\psi_k^\xi \rangle + \langle  dv,d\psi_k^\xi  \rangle \\
&=& \langle dv , \psi_k^\xi \rangle + \langle v, d\psi_k^\xi \rangle + \langle du, d\psi_k^\xi \rangle - \langle du^\xi, d\psi_k^\xi \rangle \\
&=& \langle \CL(u), \psi_k^\xi \rangle \, dt + \langle \psi_k^\xi,  \, dW \rangle - 
\sum_j \langle \tilde{u}^\xi_j,\psi_k^\xi \rangle \, d\xi_j + \sum_j \langle \psi^\xi_{k,j},v \rangle \,
d\xi_j \\
&& - \tfrac12 \sum_{i,j} \langle \tilde{u}_{ij}^\xi , \psi_k^\xi \rangle \langle Q\sigma_i^\xi,\sigma_j^\xi \rangle \, dt + \tfrac12 \sum_{i,j} \langle \psi^\xi_{k,ij},v\rangle \langle Q\sigma_i^\xi,\sigma_j^\xi \rangle  \, dt \\
&& + \sum_j \langle \psi^\xi_{k,j}, Q\sigma_j^\xi \rangle \, dt - \sum_{i,j} \langle \psi^\xi_{k,j}, \tilde{u}_i^\xi \rangle \langle Q\sigma_i^\xi , \sigma_j^\xi \rangle \, dt
 \end{eqnarray*} 
At first we look at the $dW$-terms:
\begin{eqnarray*}
\psi_k^\xi  - \sum_j \langle \psi_k^\xi, \tilde{u}^\xi_j \rangle \sigma_j^\xi + \sum_j \langle \psi^\xi_{k,j} , v \rangle \sigma_j^\xi
  &= &\psi_k^\xi - \sum_j \left[\langle \psi_k^\xi, \tilde{u}^\xi_j \rangle - \langle v,\psi^\xi_{k,j}\rangle \right] \sigma_j^\xi \\
  &\overset{(\ref{def:A})}{=} &\psi_k^\xi - \sum_j a_{kj}\sigma_j^\xi \overset{(\ref{def:sigma})}{=} 0.
\end{eqnarray*}
Next we consider the drift term:
\begin{align*}
 &\langle \psi_k^\xi, \CL(u) \rangle - \tfrac12 \sum_{i,j} \langle \tilde{u}_{ij}^\xi, \psi_k^\xi \rangle \langle Q\sigma_i^\xi, \sigma_j^\xi \rangle + \tfrac12 \sum_{i,j} \langle \psi^\xi_{k,ij},v \rangle \langle Q\sigma_i^\xi, \sigma_j^\xi \rangle \\
&- \sum_{i,j} \langle \psi^\xi_{k,j},\tilde{u}_i^\xi \rangle \langle Q\sigma_i^\xi, \sigma_j^\xi \rangle + \sum_j \langle \psi^\xi_{k,j} ,Q\sigma_j^\xi \rangle - \sum_j \langle \psi_k^\xi, \tilde{u}_j^\xi \rangle f_j(\xi) + \sum_j \langle \psi^\xi_{k,j},Q\sigma_j^\xi \rangle \\
&= \sum_{i,j} \left[ \tfrac12\langle v, \psi^\xi_{k,ij} \rangle - \langle \psi^\xi_{k,j} , \tilde{u}^\xi_i \rangle -\tfrac12 \langle \psi_k^\xi, \tilde{u}_{ij}^{\xi} \rangle \right] \langle Q\sigma^\xi_i,\sigma^\xi_j\rangle 
 \\& +  \sum_j \langle Q \psi^\xi_{k,j},  \sigma_j^\xi\rangle +\langle \psi_k^\xi, \CL(v+\tilde{u}^\xi)\rangle
 - \sum_j \left[ \langle \psi_k^\xi, \tilde{u}_j^\xi \rangle - \langle \psi^\xi_{k,j} , v \rangle \right] f_j(\xi) 
  \\& \overset{(\ref{def:f})}{=} 0
\end{align*}
This completes the proof that $\xi$ is indeed a semimartingale.
\end{proof}


\subsection{Approximate stochastic ODE for the droplet's motion}

In this section we want to analyze the exact equation for the droplet's motion and its approximation
in terms of $\eps$. We start with splitting the ansatz (\ref{e:ansatz}) into its deterministic part and extra stochastic terms given by a process $\CA_s$
\begin{align}
d\xi_r = \sum_i A_{ri}^{-1}(\xi) \langle \CL(v+\tilde{u}^\xi), \psi^\xi_i\rangle \, dt + d\CA^{(r)}_t,
\end{align}
where due to the definitions (\ref{def:sigma}) and (\ref{def:f}) the stochastic processes $\CA_t^{(r)}$ are given by
\begin{align}
\label{eq:dAt}
d\CA^{(r)}_t := &\sum_{i,j,k} A_{ri}^{-1}(\xi) \left[ \tfrac12\langle v, \psi_{i,jk}^{\xi} \rangle - \langle \psi^\xi_{j,k} , \tilde{u}^\xi_i \rangle -\tfrac12 \langle \psi^\xi_k, \tilde{u}_{ij}^{\xi} \rangle \right] \langle \CQ\sigma^\xi_i,\sigma^\xi_j \rangle \, dt \nonumber \\
&+ \sum_i A_{ri}^{-1}(\xi) \sum_j \langle \CQ \psi^\xi_{i,j},  \sigma_j^\xi\rangle \, dt
+ \sum_i A_{ri}^{-1}(\xi) \langle \psi^\xi_i, dW \rangle.
\end{align}
In this section let us first show that the $\xi_i$ are driven by a noise term of the type $\langle \tilde{u}_i^\xi,dW \rangle$, 
which means that we project the Wiener process to the slow manifold.
We also give bounds on the drift $f(\xi)$ and the diffusion $\sigma(\xi)$. 

%
In view of theorem \ref{H-Sj} we have
\begin{equation*}
\frac{\tilde{u}_i^\xi}{\Vert \tilde{u}_i^\xi  \Vert} =
\sum_k \left\langle \frac{\tilde{u}_i^\xi}{\Vert \tilde{u}_i^\xi  \Vert} , \psi^\xi_k \right\rangle 
\psi^\xi_k + \CO(\exp),
\end{equation*}
where $\psi^\xi_k$ denotes the eigenfunctions corresponding to the small eigenvalues of $\CL^\xi$. (see theorem \ref{LinCH}). 
Using $\Vert \tilde{u}^\xi_i \Vert \leq C \rho$ we get by multiplying
\begin{equation*}
\tilde{u}_i^\xi =
\sum_k \left\langle \tilde{u}_i^\xi , \psi^\xi_k \right\rangle 
\psi^\xi_k + \CO(\exp) = \sum_k b_{ik}  \psi_k^\xi + \CO(\exp).
\end{equation*}
By rotating the eigenfunctions $\psi_k^\xi$ with an orthonormal matrix $Q$ we can introduce a new coordinate system $\bar{\psi}_k^\xi$ of eigenfunctions in such a way that $\tilde{u}^\xi_1 \parallel \bar{\psi}_1^\xi$
and the corresponding matrix defined by $\bar{b}_{ij} = \langle \tilde{u}^\xi_i, \bar{\psi}_j^\xi \rangle$
is an almost diagonal matrix and the same holds true for its inverse. \\
Hereby, $Q$ will be uniquely defined by rotating the rows of $B$ such that
\begin{align}
\label{MatForm}
\bar{B} = Q B = \begin{pmatrix}
\bar{b}_{11} & 0 \\
\bar{b}_{21} & \bar{b}_{22}
\end{pmatrix}.
\end{align}
With respect to the new coordinate system we then have
\begin{align*}
\tilde{u}_i^\xi = \sum_k \langle \tilde{u}_i^\xi, \bar{\psi}^\xi_k \rangle \bar{\psi}^\xi_k + \CO(\exp)
= \sum_k \bar{b}_{ik} \bar{\psi}^\xi_k + \CO(\exp).
\end{align*}
\begin{lemma}
\label{lem:matrix}
Consider the matrix $\bar{A}(\xi) \in \R^{2 \times 2}$ given by
\begin{equation*}
\bar{A}_{kj}(\xi) = \bar{Z}^0_{kj} + \bar{Z}^1_{kj}(v) :=\langle \bar{\psi}^\xi_k, \tilde{u}_j^{\xi}\rangle -
\langle v, \bar{\psi}_{k,j}^{\xi}\rangle
\end{equation*}
Then, as long as $\Vert v \Vert \leq C \eps^{1+\kappa}$ for some $\kappa > 0$ and $0 < \eps < \eps_0$,
$\bar{A}(\xi)$ is invertible and its inverse $\bar{A}^{-1}(\xi)$ can be estimated by
\begin{equation*}
\bar{A}_{kj}^{-1}(\xi) = \|\tilde{u}^\xi_k\|^{-1} \delta_{kj} + \CO(1).
\end{equation*}
\end{lemma}

Note that the same statement holds without the bar also for the matrix $A(\xi)$.

\begin{proof}
By \cite{CHBoundary} we have
\begin{equation}
\label{ineq:Inv}
\langle \tilde{u}_i^\xi, \tilde{u}_j^\xi \rangle = C_0^2 \rho^2 \, \delta_{ij} + \CO(\rho^{3}) 
+ \CO (\eps \rho^{-1}) + \CO(\exp)
\end{equation}
and therefore $\langle \tilde{u}_i^\xi, \tilde{u}_j^\xi \rangle$ defines for small $\rho$ an
almost diagonal, invertible matrix of order $\CO(1)$. 
Moreover, $ C_0^2 \rho^2 =  \|\tilde{u}^\xi_k\|^2$.

In the comment to theorem \ref{LinCH} we proved
the link
\begin{equation*}
\langle \tilde{u}_i^\xi, \tilde{u}_j^\xi \rangle = \left( Z^0 \cdot (Z^0)^T \right)_{ij} + \CO(\exp),
\end{equation*}
where we only needed that the basis $\psi^\xi_i$ is orthonormal. Since the orthonormal transformation
$Q$ does not change this property, we similarly obtain
\begin{equation}
\label{eq:MatId}
\langle \tilde{u}_i^\xi, \tilde{u}_j^\xi \rangle = \left( \bar{Z}^0 \cdot (\bar{Z}^0)^T \right)_{ij}
+ \CO(\exp),
\end{equation}
such that invertibility of $\bar{Z}^0$ can be derived from the invertibility of 
$\left( \langle \tilde{u}_i^\xi, \tilde{u}_j^\xi \rangle \right)_{i,j}.$
On the other hand we have
\begin{align*}
\langle v, \bar{\psi}_{i,j}^\xi \rangle \leq C \Vert v \Vert \Vert \psi_{i,j}^\xi \Vert 
\leq C \eps^\kappa.
\end{align*}
From this, we see directly that $\bar{A}(\xi)$ is invertible. 
\vspace{0.3cm} \\
Using the form (\ref{MatForm}) of the matrix $\bar{Z}^0$ and relations (\ref{ineq:Inv}) and (\ref{eq:MatId}) we see that
\begin{align*}
\bar{A}_{ij} = C_0 \rho \, \delta_{ij} + \CO(\rho^2),
\end{align*}
where we neglected higher order terms.
Next, we consider the decomposition
\begin{equation*}
\bar{A}(\xi) = C_0 \rho ( I - E),
\end{equation*}
where $I$ denotes the identity matrix and $E$ is a small perturbation thereof of order $\CO(\rho)$. Then, one has by Taylor expansion
\begin{align*}
\bar{A}^{-1} &= C_0^{-1} \rho^{-1} (I - E)^{-1} = C_0^{-1} \rho^{-1} \sum_k E^k \\
&= C_0^{-1}  \rho^{-1} ( I + E + \CO(\rho^2) ) = C_0^{-1} \rho^{-1} I + \CO(1).
\end{align*}
With this the lemma is proved.
\end{proof}
\begin{lemma}
\label{lemma:var}
Under the assumptions of Lemma \ref{lem:matrix} we have
\begin{equation}
\sigma_r(\xi) = \bar{A}_{rr}^{-1}(\xi) \, \bar{\psi}_r^\xi + \CO(1) 
= C \|\tilde{u}_r^\xi\|^{-1} \tilde{u}_r^\xi + \CO(1).
\end{equation}
\end{lemma}
\begin{proof}
Immediate consequence of the definition
\begin{align*}
\sigma_r(\xi) = \sum_i \bar{A}_{ri}^{-1}(\xi) \bar{\psi}^\xi_i,
\end{align*}
where we changed the underlying coordinate system,
and the previous lemma.
Moreover, we know that $\bar{A}$ is for small $\rho$ approximately a diagonal matrix, 
so we can replace $\bar{\psi}_r^\xi$  by $\tilde{u}_r^\xi$.
\end{proof}

Next, we estimate the magnitude of the drift term $f$ in terms of $\eps$.

\begin{lemma}
\label{lemma:drift}
Under the assumptions of Lemma \ref{lem:matrix} we have
\[
\vert f(\xi , v ) \vert \leq C \eps^{-1} \eta_1.
\]
\end{lemma}

\begin{proof}
We need to estimate all $dt$-terms in the definition \eqref{eq:dAt}. Using Lemma \ref{lemma:var}  for estimating the variance $\sigma$ we derive
\begin{align*}
\vert \langle \CQ \psi_{i,j}^\xi, \sigma_j \rangle \vert \leq C \rho^{-1} \eps^{-1} \eta_1
\end{align*}
and 
\begin{align*}
\left\vert \left[ \tfrac12\langle v, \psi_{i,jk}^{\xi} \rangle - \langle \psi^\xi_{j,k} , \tilde{u}^\xi_i \rangle -\tfrac12 \langle \psi^\xi_k, \tilde{u}_{ij}^{\xi} \rangle \right] \langle \CQ\sigma^\xi_i,\sigma^\xi_j \rangle \right\vert \leq C \rho^{-2} \eps^{-1} \eta_1,
\end{align*}
where we used the estimates
 $\Vert \tilde{u}_i^\xi \Vert = \CO(1), \Vert \tilde{u}_{ij}^\xi \Vert = \CO(\eps^{-1/2})$, $\Vert \psi_{i,j}^\xi \Vert = \CO(\eps^{-1}),
\Vert \psi_{i,jk}^\xi \Vert = \CO(\eps^{-3/2})$, which will be derived in section 5, cf. Lemma \ref{lem:estH-1}.
Combining this with the estimate of  $\bar{A}_{ri}^{-1}(\xi)$ from Lemma \ref{lem:matrix} shows that the estimate holds true.
\end{proof}

\begin{remark}{(It\^o-Stratonovich-correction)}
\label{rem:ItoStra}
Let us take a closer look at \eqref{eq:dAt}. After some calculation,
basically redoing the computation that led to \eqref{def:sigma} and
\eqref{def:f} in the Stratonovich sense and thereby leaving out It\^o
corrections, one can show that with Stratonovich differentials
\[ 
\sum_j A_{kj}(\xi,v) \circ d\xi_j = \langle \psi_k^\xi, \CL( v + \tilde{u}^\xi ) \rangle \, dt + \langle \psi_k^\xi, \circ \, dW \rangle.
\]
Thus, we can solve for $ \circ \, d\xi_j$ and obtain also for the It\^o differential 
\[
d\xi_k = \CO(\exp) \, dt + \sum_j A_{kj}^{-1}(\xi,v) 
\langle \psi_j^\xi, \circ \, dW \rangle.
\]
which is (up to some exponentially small error) the projection of the Wiener process $W$ onto the slow manifold $\tilde{\CM}_\rho^\eps$ of droplets.
\end{remark}


\section{Stochastic Stability}
\label{sec:stab}

For the stochastic stability we derive bounds for the distance from the slow manifold given by $v$.
First we give a result in $H^{-1}$ and then extend it to $L^2$.

\subsection{\boldmath$H^{-1}$-bounds}

Recall that we splitted the solution via Fermi coordinates
\begin{align*}
u(t) = \tilde{u}^{\xi(t)} + v(t)
\end{align*}
with the orthogonality condition $v(t) \perp  \psi^\xi_i(t)$ in ${H^{-1}(\Omega)}$ for $i=1,2$.
In the following we always assume that we are working on times such that $\xi(t)\in\Omega_{\rho+\delta}$ 
so that everything is well defined.

Writing (\ref{eq}) in the form $du = \CL(u) \, dt +  dW$ and expanding gives
\begin{align}
\label{eq:4.1}
du = \left[ \CL(\tilde{u}^\xi) + \CL^\xi v + \CN(\tilde{u}^\xi,v) \right] \, dt +  dW,
\end{align}

and on the other hand we have
\begin{align}
\label{eq:4.2}
du = d\tilde{u}^\xi + dv = \sum_j \tilde{u}^\xi_j d\xi_j + \sum_{i,j} \tilde{u}^\xi_{ij} \langle \CQ \sigma^\xi_i, \sigma^\xi_j \rangle \, dt + dv.
\end{align}

Here we used the definitions
\begin{eqnarray*}
\CL (w) &=& - \Delta \left( \eps^2 \Delta w - F'(w) \right), \\
\CL^\xi w &=& - \Delta \left( \eps^2 \Delta w - F''(\tilde{u}^\xi) w \right), \\ 
\CN(y,z) &=& - \Delta \left(- F'(y+z) + F'(y) + F''(y)z \right). 
\end{eqnarray*}

In the case $F(u) = \tfrac14 (u^2 -1 )^2$ we have
\begin{equation*}
\CN(\tilde{u}^\xi ,v ) = - \Delta (- 3 \tilde{u}^\xi v^2 - v^3). 
\end{equation*}

From Theorem \ref{thm:slowmf} (iv) we have for the residual
\begin{align}
\label{eq:4.3}
\CL(\tilde{u}^\xi) = \sum_j c^\xi_j u^\xi_j =\mathcal{O}(\exp).
\end{align}

Solving (\ref{eq:4.1}) and (\ref{eq:4.2}) for $dv$ and substituting (\ref{eq:4.3}), we obtain the equation
for the flow orthogonal to the slow manifold.
\begin{lemma}
Consider a solution $u(t) = \tilde{u}^{\xi(t)} + v(t)$ with $v(t) \perp_{H^{-1}} \psi_i(t)$ for $i=1,2$ and $\xi(t)$ being the diffusion process given by (\ref{def:sigma}) and (\ref{def:f}), then
\begin{align}
\label{eq:dv}
dv = &\Big( \sum_j c_j^\xi u_j^\xi + \CL^\xi v + \CN(\tilde{u}^\xi,v) \Big) dt
+  dW \nonumber \\
&- \sum_j  \tilde{u}_j^{\xi} d\xi_j 
-\tfrac12 \sum_{i,j}  \tilde{u}_{ij}^{\xi}  \langle \CQ\sigma^\xi_j,\sigma^\xi_i\rangle dt.
\end{align}
\end{lemma}

Let us now turn to the estimate of $\| v \|^2_{H^{-1}}$. We first notice that It\^o calculus gives
$$
d\| v \|^2_{H^{-1}} = 2 \, \langle v , dv \rangle +  \langle dv,dv
\rangle.
$$
Since $d\xi = b(\xi) \, dt + \langle \sigma, dW \rangle$ and
$$
\langle dW, dW \rangle = \text{trace}(\CQ) \, dt = \eta_0 \, dt,
$$
again by It\^o calculus we derive
\begin{align*}
\langle dv, dv \rangle  &=
   \Big\langle \sum_j \tilde{u}_j^\xi \cdot d\xi_j, \sum_j \tilde{u}_j^\xi \cdot d\xi_j \Big\rangle
   - 2 \Big\langle  dW, \sum_j \tilde{u}_j^\xi \cdot d\xi_j \Big\rangle
   + \Big\langle  dW,  dW \Big\rangle \\
 &= \eta_0 \, \mathrm{dt}
    + \sum_{i,j} \langle \tilde{u}^\xi_i, \tilde{u}^\xi_j \rangle \langle \CQ \sigma_i^\xi,\sigma_j^\xi \rangle \, dt
    - 2 \sum_j \langle \tilde{u}^\xi_j , \CQ \sigma_j \rangle \, dt.
\end{align*}
Using the notations $\| \partial_\xi \tilde{u}^\xi \| = \max \| \tilde{u}^\xi_i \|$ and 
$\| \sigma \| = \max \| \sigma^\xi_i \|$ we have
\begin{equation}
\label{est:dvdv}
\langle dv, dv \rangle 
\leq 
\left( \eta_0 + \| \partial_\xi \tilde{u}^\xi \|^2\  \| \sigma \|^2 
 \| \CQ \| 
 + 2 \, \| \partial_\xi \tilde{u}^\xi \|  \| \sigma \|  \| \CQ \| \right) dt
 = 
\CO ( \eta_0 ) \, dt,
\end{equation}
where we used that $ \| \partial_\xi \tilde{u}^\xi \| = \CO(1)$,
$\| \sigma \| = \CO(1)$ by Lemma \ref{lemma:var} and $\norm{ \CQ }_{L(H^{-1})}= \eta_1 \leq \eta_0$.

Next, we investigate the more involved term
\begin{align}
\label{eq:vdv}
\langle v, dv \rangle = 
&\Big[ \sum_j c_j^\xi \, \langle u^\xi_j, v \rangle 
+ \langle \CL^\xi v, v \rangle
+ \langle \CN(\tilde{u}^\xi,v), v \rangle \Big] \, dt \nonumber \\
&- \sum_j \langle \tilde{u}^\xi_j, v \rangle \, d\xi_j
 - \frac12 \sum_{i,j} \langle \tilde{u}^\xi_{ij}, v \rangle  \langle \CQ \sigma_i, \sigma_j \rangle \, dt
 + \langle v, dW \rangle.
\end{align}
We start with deriving a bound for the nonlinear term $ \langle \CN(\tilde{u}^\xi ,v ), v \rangle $ 
by using spectral information for the linearized Cahn-Hilliard operator $\CL^\xi $ in $H^{-1}(\Omega)$. 
Here, it is useful that the spectral theory of the Cahn-Hilliard equation in $H^{-1}$ coincides with
the Allen-Cahn operator in $L^2$ (Remark \ref{rem:MCAC}).
\begin{lemma}
\label{est: Nonlinearity}
For $u=u^\xi+v$ with $\norm{v}_{H^{-1}(\Omega)} < c_0 \eps^{4}$ for some fixed sufficiently small $c_0>0$ we have
\begin{equation*}
\langle \CL^\xi v , v \rangle_{H^{-1}(\Omega)} + \langle \CN(\tilde{u}^\xi,v),v \rangle_{H^{-1}(\Omega)}
\leq - C \eps \norm{v}_{H^{-1}(\Omega)}^2.
\end{equation*}
\end{lemma}

\begin{proof}
Let $\gamma_1, \gamma_2, \gamma_3 \geq 0$ with $\sum_i \gamma_i = 1$. First, we notice that we have
\begin{equation*}
\langle \CL^\xi v,v \rangle_{H^{-1}} = \langle \eps^2 \Delta v + F''(\tilde{u}^\xi) v,v \rangle_{L^2} 
\leq - \eps^2 \| \nabla v \|_{L^2}^2 + C \| v \|_{L^2}^2,
\end{equation*}
where we performed integration by parts. Together with the spectral information of Theorems \ref{LinCH}
and \ref{LinAC} for the linearized Cahn-Hilliard operator in $H^{-1}$ and the linearized non-local Allen-Cahn
operator in $L^2$ we derive
\begin{align}
\label{ineq:Split}
\langle \CL^\xi v,v \rangle_{H^{-1}} &= \sum_i \gamma_i \langle \CL^\xi v,v \rangle_{H^{-1}} \nonumber \\
&\leq - C \gamma_1 \eps \| v \|_{H^{-1}}^2 - C \gamma_2 \eps^2 \| v \|_{L^2}^2 
- \gamma_3 \eps^2 \| \nabla v \|_{L^2}^2 + C \gamma_3 \| v \|_{L^2}^2 \nonumber \\
&\leq - c \eps \| v \|_{H^{-1}}^2 - c \eps^2 \| v \|_{L^2}^2 - c \eps^4 \| v \|_{H^1}^2,
\end{align}
where we fixed $\gamma_3 \approx \eps^2$ and absorbed the positive $L^2$-term into its negative
counterpart. \\
 As long as $\| v \|_{H^{-1}} \leq c_0 \eps^4$ we have
\begin{align*}
\langle \CN(\tilde{u}^\xi,v),v \rangle_{H^{-1}} 
&\leq C \| v \|_{L^3}^3 \leq C \| v \|_{H^{1/3}}^3 \\
&\leq C \| v \|_{H^1}^2 \| v \|_{H^{-1}} 
\leq C  c_0 \eps^4 \| v \|_{H^1}^2.
\end{align*}
Here, we used $H^{1/3}(\Omega) \hookrightarrow L^3(\Omega)$ by Sobolev embedding and interpolation of
$H^{1/3}$ between $H^{-1}$ and $H^1$. 
Combined with (\ref{ineq:Split}) we get by choosing $c_0$ sufficiently small compared to the other constants
\begin{equation*}
\langle \CL^\xi v, v \rangle_{H^{-1}} + \langle \CN(\tilde{u}^\xi,v),v \rangle_{H^{-1}}
\leq - c \eps \| v \|_{H^{-1}}^2 - c \eps^2 \| v \|_{L^2}^2 - C\eps^4 \| \nabla v \|_{L^2}^2\;.
\end{equation*}
\end{proof}
We need to control the terms of (\ref{eq:vdv}) containing inner products with first derivatives of 
$u^\xi$ and $\tilde{u}^\xi$, respectively. As $\tilde{u}^\xi_i$ can be seen as approximation of the
eigenfunctions $\psi^\xi_i$ together with the orthogonality condition (\ref{eq:Fermi}), we may
assume that up to some exponentially small error $v \perp \frac{\partial u^\xi}{\partial \xi_i}$.

\begin{lemma}
\label{ApproxOrtho}
Let $v$ be as in Proposition \ref{thm:Fermi}. Then we have
\begin{align*}
\Big\langle \frac{\partial u^\xi}{\partial \xi_i}, v \Big\rangle_{H^{-1}} = \CO(\exp )\Vert v \Vert_{H^{-1}}, \quad i = 1,2,
\end{align*}
and the same holds true for $u^\xi$ replaced by $\tilde{u}^\xi$.
\end{lemma}
\begin{proof}
From \ref{LinCH} and \ref{H-Sj} we see that the distance of $U^\xi = \textrm{span}\{ \psi^\xi_1,\psi^\xi_2 \}$ and 
$\textrm{span}\{ \tilde{u}^\xi_1, \tilde{u}^\xi_2  \}$ is of order $\CO(\exp)$. Therefore, for some $\alpha_j \in \R$
\begin{align*}
\langle \tilde{u}^\xi_j , v \rangle = \sum_j \alpha_j \langle \psi^\xi_j, v \rangle + \langle \CO(\exp), v \rangle = \CO (\exp) \Vert v \Vert_{H^{-1}}.
\end{align*}
With $\Vert \tilde{u}^\xi_j - u^\xi_j \Vert = \CO(\exp)$ the lemma is derived.
\end{proof}

Finally, we can continue with estimating $\langle v, dv \rangle$. By Lemmata \ref{est: Nonlinearity} and
\ref{ApproxOrtho} together with the estimate for the second derivatives of $\tilde{u}^\xi$ we derive
\begin{align}
\label{est:vdv}
\begin{split}
\langle v,dv \rangle
\leq & [-C\eps \Vert v \Vert_{H^{-1}}^2 + \CO(\exp) \Vert v \Vert_{H^{-1}}
+ \CO( \eps^{-1/2} \eta_1) \Vert v \Vert_{H^{-1}}]dt
\\& + \langle v + \CO(\exp), dW \rangle.
\end{split}
\end{align}
Here, we also used that the drift term of $d \xi$ is of order $\CO(\eps^{-1})$ 
which we proved in Lemma \ref{lemma:drift}. Thereby with lemma \ref{ApproxOrtho}, 
the term $\sum_j \langle \tilde{u}_j^\xi , v \rangle d\xi_j$ 
remains exponentially small. \\
We summarize the $H^{-1}$ estimate in the following theorem:

\begin{theorem}
\label{thm:DiffIneq}
As long as $\Vert v \Vert_{H^{-1}} \leq c_0 \eps^{4}$ with $c_0>0$ from Lemma \ref{est: Nonlinearity} it holds that
\begin{align}
d \Vert v \Vert_{H^{-1}}^2 \leq \left[ C_\eps - C \eps \Vert v \Vert_{H^{-1}}^2 \right] \, dt 
+ 2 \langle v + \CO(\exp),  dW \rangle_{H^{-1}},
\end{align}

where
\begin{align*}
C_\eps =  C \eta_0 + \CO( \exp ).
\end{align*}
\end{theorem}
\begin{proof}
By (\ref{est:dvdv}) and (\ref{est:vdv}) we have
\begin{align*}
d \| v \|_{H^{-1}}^2 \leq &\Big[ - C \eps \| v \|_{H^{-1}}^2 + C \eps^{-1/2}\eta_1 
\| v \|_{H^{-1}} + C \eta_0 + \CO(\exp) \Big] \, dt \\
&+ 2 \langle v + \CO(\exp), dW \rangle.
\end{align*}
As $\eta_1\leq\eta_0$ we obtain
\begin{align*}
\eps^{-1/2} \eta_1 \Vert v \Vert_{H^{-1}} 
\leq c_0 \eps^{7/2} \eta_0  
\end{align*}
and thereby the claim.
\end{proof}
%


\subsection{Long-time stability in \boldmath$H^{-1}$}

We follow a method used in \cite{DropletAC} for the stochastic Allen-Cahn equation to show the long-time
stability with respect to the $H^{-1}$ norm. \\
Define the stopping time $\tau^\star$ as the exit time from a neighborhood of the slow manifold before
time $T$
\[
\tau^\star := \inf \left\{ t \in [0,T] : \| v(t) \|_{-1} > B \right\} 
\]
with the convention that $\tau^\star = T$, if $\| v(t) \|_{-1} \leq B$ for all $t \in [0,T]$.
Note that we neglect the case that $\xi(t)\not\in\Omega_{\delta+\rho}$ at some point.
We only need to cut with another stopping time to take care of this.

We showed in Theorem \ref{thm:DiffIneq} that $v$ satisfies a differential inequality of the form
\begin{equation}
\label{DiffIneq}
  \mathrm{d} \| v(t) \|^2_{-1} 
  \leq 
  \left[ C_\eps - a \| v(t) \|^2_{-1} \right] \mathrm{d}t 
  + 2 \left( v, \mathrm{d}W \right)
\end{equation}
for all $t \leq \tau^\star$, provided that $B \leq c_0 \eps^{4}$. 

From \cite{DropletAC} using optimal stopping of martingales, we obtain from \eqref{DiffIneq} 
\begin{equation}
\label{Ito1}
\E \| v(\tau^\star) \|^{2p}_{-1} 
\leq 
\| v(0) \|^{2p}_{-1} + C \Big[ C_\eps + \| \CQ \| \Big] \E \int_0^{\tau^\star}\|v\|^{2p-2}_{-1} \mathrm{d}s
\end{equation}
and
\begin{equation}
\label{Ito2}
a \E \int_0^{\tau^\star}\|v\|^{2p}_{-1} \mathrm{d}s 
\leq 
\frac{1}{p} \| v(0) \|^{2p}_{-1} + C \Big[ C_\eps + \| \CQ \| \Big] \E \int_0^{\tau^\star}\|v\|^{2p-2}_{-1} \mathrm{d}s.
\end{equation}
We define now $q$ and assume the following
\begin{equation}
\label{e:defq}
q := \frac{C_\eps + \| \CQ \|}{a} \ll 1 
\quad\quad \text{and} \quad\quad 
\| v(0) \|^2 \leq q \ll B^2.
\end{equation}
Via an induction argument we derive
\[
\tfrac{1}{p} \, \E \| v(\tau^\star) \|^{2p} 
\leq
Cq^p + Caq^pT
\]
as $C_\eps \leq aq$. Chebychev's inequality finally yields
\begin{align}
\label{Cheb}
\prob \left( \tau^\star < T \right) &= \prob \left( \| v(\tau^\star) \| \geq B \right) 
\leq B^{-2p} \cdot \E \| v(\tau^\star) \|^{2p} \nonumber \\
&\leq CB^{-2p} \Big[ q^p + aq^p T \Big] = C \left( \tfrac{q}{B^2} \right)^p + C a \left( \tfrac{q}{B^2} \right)^p T.
\end{align}
With this, we can prove the following theorem:
\begin{theorem}
\label{thm:H-1stab}
For a solution $u=u^\xi+v$ with $\xi\in\Omega_{\rho+\delta}$ and $v\perp \psi^\xi_j$ 
consider the exit time 
\[
\tau^\star = \inf \left\{ t \in [0,T_\eps] \, : \, \| v(t) \|_{-1} > c_0 \eps^4 \right\},
\]
with $T_\eps = \eps^{-N}$ for any fixed large $N > 0$ and $c_0>0$ from Lemma \ref{est: Nonlinearity}. 
Fix with $\nu<c_0$
\[
\| v(0) \|_{-1} \leq \nu \eps^{4}.
\]
Also, assume that the noise strength satisfies
\[
\eta_0 \leq C \eps^{9 + \tilde{k}},
\]
for some $\tilde{k} > 0$ very small. 
Then the probability $ \prob \left( \tau^\star < T_\eps \right)$ is smaller than any power  of $\eps$, 
as $\eps$ tends to $0$. 
\end{theorem}

And thus for very large time scales with high probability the solution stays 
close to the slow manifold $\tilde{\CM}_\rho^\eps$.
Unless the droplet gets close to the boundary, i.e. $\xi(t)\not\in\Omega_{\delta+\rho}$.

\begin{remark}
 In Remark \ref{rem:ItoStra} we saw that for $\eta_0$ being polynomial in $\eps$ 
the position $\xi$ of the droplet is moving like a diffusion process driven by a 
Wiener process of strength $\sqrt{\eta_0}$ which is 
 multiplied by a diffusion coefficient of order $\CO(1)$. Thus due to scaling, we would expect 
 that the droplet hits the boundary of the domain after time scales of order larger that $1/\eta_0$. 
 
 Thus the stability result tells us that with overwhelming probability the solution moves along 
 the deterministic slow manifold until it hits the boundary of the domain.
 \end{remark}

\begin{proof}
The statement follows directly from (\ref{Cheb}) if 
$\frac{q}{B^2} = \CO (\eps^{\tilde{k}})$.

Indeed, using the definition of $C_\eps$, $a = \CO(\eps)$ and $B = \CO(\eps^{4})$, we have
\[
q := \frac{C_\eps + \| \CQ \|}{a} \leq C\frac1{\eps} \Big[\eta_0  + \CO(\exp) \Big],
\]
since $\eta_1 \leq \eta_0$.
And therefore we finally get
\[
q/B^2 \leq C \eps^{-9}\Big[ \eta_0 + \CO(\exp) \Big]= \CO( \eps^{\tilde{k}}).
\]
\end{proof}

We can 
also treat smaller neighboorhoods of the slow manifold,
by making the size  of the  noise even smaller.

We can take the radius $B=\eps^m$ and the noise  strength 
$\eta_0=\eps^{2m+1+\tilde\kappa}$. 
If $m>4$, then we can follow exactly the same proof, as all estimates 
needed just $B\leq c_0\eps^4$. We obtain:

\begin{theorem}
\label{thm:H-1stabnew}
For a solution $u=u^\xi+v$ with $\xi\in\Omega_{\rho+\delta}$ and $v\perp \psi^\xi_j$ 
consider the exit time 
\[
\tau^\star = \inf \left\{ t \in [0,T_\eps] \, : \, \| v(t) \|_{-1} > \eps^m \right\},
\]
with $T_\eps = \eps^{-N}$ for any fixed large $N > 0$ and $m>4$. 
Fix with $\nu<1$
\[
\| v(0) \|_{-1} \leq \nu \eps^{4}.
\]
Also, assume that the noise strength satisfies
\[
\eta_0 \leq C \eps^{2m+1+ \tilde{\kappa}},
\]
for any $\tilde{\kappa} > 0$ small. 
Then the probability $ \prob \left( \tau^\star < T_\eps \right)$ is smaller than any power  of $\eps$, 
as $\eps$ tends to $0$. 
\end{theorem}

\subsection{Estimates in \boldmath$L^2$-norm}

We want to extend the stability result to the $L^2$-norm. As there are no bounds of the
linearized Cahn-Hilliard operator in $L^2$, we will rely on the results of the previous section.

Recall \eqref{eq:dv},
\begin{align*}
dv = &\Big( \sum_j c_j^\xi u_j^\xi + \CL^\xi v + \CN(\tilde{u}^\xi,v) \Big) dt + dW \\
&- \sum_j  \tilde{u}_j^{\xi} d\xi_j 
-\tfrac12 \sum_{i,j}  \tilde{u}_{ij}^{\xi}  \langle \CQ\sigma^\xi_j,\sigma^\xi_i\rangle dt,
\end{align*}
where 
\begin{align*}
\CL^\xi v + \CN(\tilde{u}^\xi,v) = - \eps^2 \Delta^2 v
 + \Delta \left[ f(\tilde{u}^\xi+v) - f(\tilde{u}^\xi) \right].
\end{align*}

As our object of interest is the $L^2$-norm of $v$ we consider the relation
\begin{align}
d \Vert v \Vert_{L^2}^2 = 2 \left( v, dv \right)_{L^2} + \left( dv,dv \right)_{L^2}.
\end{align}
Recall that we denote the $L^2$ inner product by 
$\left(\cdot,\cdot \right)$ and the $H^{-1}$ inner product by
$\langle \cdot,\cdot \rangle$.

By series expansion of $W$ we obtain
\begin{align*}
\left( \tilde{u}_j^{\xi}, dW \right) \langle \sigma_i, dW \rangle &= 
\sum_{k=0}^\infty \alpha_k \left( \tilde{u}_j^{\xi}, e_k \right) \, d\beta_k
\sum_{l=0}^\infty \alpha_l \langle \sigma_i, e_l \rangle \, d\beta_l  \\
&= \sum_{k=0}^\infty \alpha_k \left( \tilde{u}_j^{\xi}, e_k \right) \alpha_k
\langle \sigma_i, e_l \rangle \, dt 
\leq \eta_0^{1/2} \eta_2^{1/2} \Vert \tilde{u}_j^{\xi} \Vert_{L^2}
\Vert \sigma_i \Vert_{H^{-1}} dt \\
 &= \CO( \eps^{-1} \eta_0 + \eta_2) \, dt,
\end{align*}
where we used the $H^{-1}$ estimate of $\sigma$ from the previous section and
$\Vert \tilde{u}_j^{\xi} \Vert_{L^2} = \CO( \eps^{-1/2})$, as the derivative $\tilde{u}^\xi_j$ is $\mathcal{O}(\eps^{-1})$ on a set of order $\mathcal{O}(\eps)$.

Thus, for the It\^o correction term we have
\begin{align*}
\left( dv, dv \right) &= \sum_{i,j} \left( \tilde{u}_i^{\xi},\tilde{u}_j^{\xi} \right)
\langle \CQ \sigma_i, \sigma_j \rangle \, dt - 2 \sum_i \left( \tilde{u}_i^{\xi}, dW
\right) \langle \sigma_i, dW \rangle + \left(dW, dW \right) \\
&\leq C \Vert \tilde{u}_i^{\xi} \Vert^2_{L^2} \Vert \sigma_i \Vert_{H^{-1}}^2 \eta_1
+ \CO( \eps^{-1} \eta_0 + \eta_2) \, dt + \trace(-\Delta \CQ) \, dt \\
&= \CO( \eps^{-1} \eta_0 + \eta_2) \, dt
\end{align*}
as $\eta_1 \leq \eta_0$.
\\

Next, we study the mixed term $\left( v, dv \right)$. By \eqref{eq:dv} we have
\begin{align*}
\left( v, dv \right) = 
&\left[ \sum_i (c_i-b_i) \left( \tilde{u}_i^{\xi},v \right) \right] \, dt
+ \left[ \left(v,dW \right) - \sum_i \left(\tilde{u}_i^{\xi},v \right)
\langle \sigma_i, dW \rangle \right] \\
&\left[ \left( - \eps^2 \Delta^2 v
 + \Delta ( f(\tilde{u}^\xi+v) - f(\tilde{u}^\xi) \right),v ) \right] \,dt
- \frac12 \sum_{i,j} \left( \tilde{u}_{ij}^{\xi},v \right) \langle \CQ \sigma_i, 
\sigma_j \rangle \, dt \\
&=: T_1 + T_2 + T_3 + T_4.
\end{align*}
For the martingale term we see that 
\begin{align*}
T_2 &= \langle \CO(\eps^{-1/2} \Vert v \Vert_{L^2}), dW \rangle_{H^{-1}} 
+ \langle (-\Delta)^{1/2}v, d(-\Delta)^{1/2} W \rangle_{H^{-1}} \\
&= \langle \CO(\eps^{-1/2} \Vert v \Vert_{L^2}), dW \rangle_{H^{-1}}
+ \langle \CO( \Vert v \Vert_{L^2}), d (-\Delta)^{1/2} W \rangle_{H^{-1}},
\end{align*}
where the $\CO$-terms are all bounded in $H^{-1}$.

For $T_4$ we have
\begin{align*}
T_4 = \CO \left( \eps^{-3/2} \eta_1 \Vert v \Vert_{L^2} \right).
\end{align*}
$c$ is by definition exponentially small and we established in section 3.3 that the drift term $b$ is of order $\CO(\eps^{-1} \eta_1)$. Thus, we have 
\begin{align*}
T_1 = \CO( \eps^{-1} \eta_1 \Vert \partial_\xi \tilde{u}^\xi \Vert_{L^2} \Vert v \Vert_{L^2} ) = \CO( \eps^{-3/2} \eta_1 \Vert v \Vert_{L^2} ).
\end{align*}
It remains to estimate the term $T_3$ involving the nonlinearity. Integration by parts 
immediately yields 
\[ 
\left( -\eps^2 \Delta^2 v,v \right) = - \eps^2 \Vert \Delta v \Vert^2_{L^2},
\]
which is a good term for the estimate. 
We continue with the other terms in $T_3$
\begin{align*}
\left( \Delta \left[ f(\tilde{u}^\xi+v) - f(\tilde{u}^\xi) \right],v \right)
&=\left(v,\Delta \left[ v^3 + 3 \tilde{u}^\xi v^2 + 3 (\tilde{u}^\xi)^2 v \right] \right) \\
&\leq C \left[ \Vert v \Vert_{L^6}^3 + \Vert v \Vert_{L^4}^2 \right] \Vert v \Vert_{H^2}
+ C \Vert v \Vert_{L^2} \Vert v \Vert_{H^2}.
\end{align*}
For the higher order powers we obtain 
by Sobolev embedding and interpolation inequalities 
\begin{align*}
C \left[ \Vert v \Vert_{L^6}^3 + \Vert v \Vert_{L^4}^2 \right] \Vert v \Vert_{H^2}
&\leq C \left[ \Vert v \Vert_{H^{1/2}}^2 + \Vert v \Vert_{H^{2/3}}^3 \right] 
\Vert v \Vert_{H^2} \\
&\leq C \left[ \Vert v \Vert_{L^2}^{3/2} \Vert v \Vert_{H^2}^{1/2}
+ \Vert v \Vert_{L^2}^2 \Vert v \Vert_{H^2} \right] \Vert v \Vert_{H^2} \\
& \leq C \left[ \Vert v \Vert_{H^{-1}}^{2\gamma} \Vert v \Vert_{L^2}^{3/2 - 3 \gamma}
\Vert v \Vert_{H^2}^{1/2 + \gamma} + \Vert v \Vert_{L^2}^2 \Vert v \Vert_{H^2} \right]
\Vert v \Vert_{H^2}.
\end{align*}
By choosing $\gamma = 1/2$ we finally derived
\begin{align*}
\left( \Delta \left[ f(\tilde{u}^\xi+v) - f(\tilde{u}^\xi) \right],v \right) 
&\leq C \left[ \Vert v \Vert_{H^{-1}} + \Vert v \Vert_{L^2}^2 \right] \Vert v \Vert_{H^2}^2 
+ C \Vert v \Vert_{L^2} \Vert v \Vert_{H^2} 
.
\end{align*}
The crucial term is the quadratic term in $v$, here we have to use the bound in $H^{-1}$. 
By interpolation and Young inequality
\begin{align*}
 C\Vert v \Vert_{L^2} \Vert v \Vert_{H^2} 
 &\leq C  \Vert v \Vert_{H^{-1}}^{2/3} \Vert v \Vert_{H^2}^{4/3} 
 =  C \eps^{-4/3} \Vert v \Vert_{H^{-1}}^{2/3} \eps^{4/3} \Vert v \Vert_{H^2}^{4/3} \\
 &\leq   C \eps^{-4} \Vert v \Vert_{H^{-1}}^2    +  \frac12\eps^{2} \Vert v \Vert_{H^2}^{2}
\end{align*}
Combining all estimates we have
\[
T_3 \leq -   \left[ \frac12\eps^2 
- C  \Vert v \Vert_{H^{-1}} - C\Vert v \Vert_{L^2}^2 \right]  \Vert \Delta v \Vert^2_{L^2}
+ C \eps^{-4} \Vert v \Vert_{H^{-1}}^2.
\]
Recall that in the preceding section we established  an optimal radius with respect to the $H^{-1}$--norm of order $\CO(\eps^4)$.
We will add a condition on the $L^2$ -- radius such that in the last estimate of the
nonlinearity the leading order of the $H^2$ -- terms is $\CO(\eps^2)$.
\begin{definition}
For $k>0$ and $m>4$ and some given large time $T_\eps$ we define the stopping time
\begin{align}
\label{eq:L2stopping}
\tau_\eps = \inf \left\{ t \in [0,T_\eps ] : \Vert v(t) \Vert_{H^{-1}} > \eps^m 
\text{ or }
\Vert v(t) \Vert_{L^2} > \eps^{k+1} \right\}.
\end{align}
Obviously, we set $\tau^\eps = T_\eps$ if none of the above conditions are fulfilled.
Again, we assume that the solution is well-defined up to $T_\eps$.
\end{definition}

Later, as we establish stability, we will need to refine the parameter $k$ defining
the $L^2$ -- radius. For now, up to the stopping time $\tau_\eps$, we have shown that
for small $\eps$
\[ 
T_3 \leq - c \eps^2 \Vert v \Vert_{H^2}^2 + C\eps^{2m-4}.
\]
Next, we use that by Poincare $\Vert v \Vert_{L^2} \leq \Vert \Delta v \Vert_{L^2}$ and $\eta_1 \leq \eta_0$
to finally get the following estimate for $d \Vert v \Vert_{L^2}^2$.
\begin{lemma}
\label{lemma:L2est}
If $k \geq 0$ and $t \leq \tau_\eps$, with $\tau_\eps$ given by \eqref{eq:L2stopping},
then for some $c>0$ the following relation holds true
\begin{align}
\label{eq:L2est}
d \Vert v \Vert_{L^2}^2 + c\eps^2 \Vert v \Vert_{L^2}^2 \, dt 
= K_\eps \, dt + \langle Z_\eps, dW \rangle_{H^{-1}} + \langle \Psi_\eps, d(-\Delta)^{1/2}W \rangle_{H^{-1}},
\end{align}
where
\begin{align*}
K_\eps = \CO( \eps^{2m-4} + \eps^{k-1/2} \eta_0+ \eps^{-1} \eta_0 + \eta_2)
\end{align*}
and
\begin{align}
\label{def:L2est}
\Vert Z_\eps \Vert_{H^{-1}}^2 = \CO(\eps^{-1} \Vert v \Vert_{L^2}^2), \quad
\Vert \Psi_\eps \Vert_{H^{-1}}^2 = \CO( \Vert v \Vert_{L^2}^2).
\end{align}
\end{lemma}

As in the $H^{-1}$ case we will derive higher moments in the subsequent section and
show stability.

\subsection{Long-time stability in \boldmath$L^2$}

Under the assumptions of Lemma \ref{lemma:L2est} we estimate for any $p>1$ the $p$-th
moment of $\Vert v \Vert_{L^2}^2$. Here we follow again the method used in \cite{DropletAC}
closely and therefore spare the reader some of the details of the derivation. By It\^o
calculus we obtain
\begin{align*}
d\Vert v \Vert_{L^2}^{2p} = p \Vert v \Vert_{L^2}^{2p-2} \, d\Vert v \Vert_{L^2}^2
+p(p-1) \Vert v \Vert_{L^2}^{2p-4} \left[ d\Vert v \Vert_{L^2}^2 \right]^2.
\end{align*}
We briefly comment on estimating the It\^o correction. Using \eqref{eq:L2est} yields
\begin{align}
\label{eq:ItoL2}
\left[ d\Vert v \Vert_{L^2}^2 \right]^2 = \langle  Z_\eps, \CQ Z_\eps \rangle  \, dt
+ \langle \Psi_\eps, - \Delta \CQ \Psi_\eps \rangle \, dt
+ 2 \langle  Z_\eps, dW \rangle \langle  \Psi_\eps, d(-\Delta)^{1/2} W \rangle
\end{align}
and by series expansion we see that
\begin{align*}
\langle Z_\eps, dW \rangle \langle \Psi_\eps, d(-\Delta)^{1/2} W \rangle
&=\sum \alpha_k^2 \langle Z_\eps,e_k \rangle \langle \Psi_\eps,(-\Delta)^{1/2}e_k \rangle
\,dt \\
&\leq \sum \alpha_k^2 \Vert e_k \Vert_{H^{-1}} \Vert e_k \Vert_{L^2} \Vert Z_\eps \Vert
\Vert \Psi_\eps \Vert \, dt \\
&\leq \Vert Z_\eps \Vert \Vert \Psi_\eps \Vert \sqrt{\eta_0 \eta_2} \\
&\leq \Vert Z_\eps \Vert^2 \eta_0 + \Vert \Psi_\eps \Vert^2 \eta_2.
\end{align*}
Therefore, by Cauchy-Schwarz, we derived
\begin{align}
\label{eq:L2Itocorr}
\left[ d\Vert v \Vert_{L^2}^2 \right]^2 \leq  
C \left[ \Vert Z_\eps \Vert_{H^{-1}}^2 \eta_0 + \Vert \Psi_\eps \Vert_{H^{-1}}^2 \eta_2 \right] \, dt.
\end{align}
Plugging \eqref{eq:L2est} and \eqref{eq:L2Itocorr} into relation \eqref{eq:ItoL2}
combined with the definitions \eqref{def:L2est} we derive the following lemma
by integrating.
\begin{lemma}
\label{lemma:induct}
Under the assumptions of Lemma \ref{lemma:L2est}, for any $p>1$ the following estimate
holds true
\begin{align*}
\E \Vert v(\tau_\eps) \Vert^{2p} + cp\eps^2 A_p 
\leq \Vert v(0) \Vert_{L^2}^{2p} + C \left[ K_\eps + \eps^{-1} \eta_0 + \eta_2 \right]
 A_{p-1},
\end{align*}
where $A_p$ is defined as 
\begin{align*}
A_p = \E \int_0^{\tau_\eps} \Vert v(s) \Vert_{L^2}^{2p} \, \mathrm{d}s.
\end{align*}
\end{lemma}

For the sake of simplicity we define 
\begin{equation}
\label{e:Defaeps}
 a_\eps = C\eps^{-2} \left[K_\eps  + \eps^{-1} \eta_0 + \eta_2 \right]
\end{equation}
and assume that the noise strength is small enough such that $a_\eps < 1$.
Note that by the definition of $K_\eps$ we thus also need $C\eps^{2m-6}<1$, 
which is true by  assumption.

Applying Lemma \ref{lemma:induct} inductively we obtain
\begin{align*}
A_p &\leq C\eps^{-2} \Vert v(0) \Vert_{L^2}^{2p} + C a_\eps A_{p-1} \\
&\leq C\eps^{-2} \Vert v(0) \Vert_{L^2}^{2p} + Ca_\eps \eps^{-2} 
\Vert v(0) \Vert_{L^2}^{2p-2} +a_\eps^2 A_{p-2} \\
&\leq \ldots \leq C\eps^{-2} \sum_{i=2}^p a_\eps^{p-i} \Vert v(0) \Vert^{2i}_{L^2}
+ C a_\eps^{p-1} A_1.
\end{align*}
Note that by \eqref{eq:L2est} we have for $t \leq \tau_\eps$
\begin{align*}
\E \int_0^t \Vert v(s) \Vert^2_{L^2} \, ds \leq C \eps^{-2} K_\eps T_\eps 
+ \eps^{-2} \Vert v(0) \Vert^2_{L^2}
\leq a_\eps T_\eps + \eps^{-2} \Vert v(0) \Vert^2_{L^2}.
\end{align*}
Hence, we derive
\begin{align}
\label{eq:Ap}
A_p \leq C\eps^{-2} \sum_{i=1}^p a_\eps^{p-i} \Vert v(0) \Vert_{L^2}^{2i}
+ C a_\eps^p T_\eps
\leq C \left[ \eps^{-2} + T_\eps \right] a_\eps^p
+ C \eps^{-2} \Vert v(0) \Vert^{2p}_{L^2}
\end{align}
for $C$ a constant depending on $p$.
\begin{lemma}
\label{lemma:L2exp}
Let $k\geq 2$ and $\tau_\eps$ as defined in \eqref{eq:L2stopping}.

If
\[
\Vert v(0) \Vert_{L^2}^2 \leq a_\eps <1
\]
then for any $p>1$ it holds true that
\begin{align*}
\E \Vert v(\tau_\eps) \Vert_{L^2}^{2p} \leq C\eps^2 \left[ \eps^{-2} + T_\eps \right] 
a_\eps^p.
\end{align*}
\end{lemma}
Note that in the previous Lemma, if $\Vert v(0) \Vert_{L^2}^2>C\eps^{k+1}$ then $\tau_\eps=0$.

\begin{proof}
By Lemma \ref{lemma:induct} and \eqref{eq:Ap} we have
\begin{align*}
\E \Vert v(\tau_\eps) \Vert_{L^2}^{2p} &\leq \Vert v(0) \Vert_{L^2}^{2p}
+ C \left[ K_\eps + \eps^{-1} \eta_0 + \eta_2 \right] A_{p-1} \\
&= \Vert v(0) \Vert_{L^2}^{2p} + C\eps^2 a_\eps A_{p-1} \\
&\leq \Vert v(0) \Vert_{L^2}^{2p} + C\eps^2 a_\eps \left[ \eps^{-2} + T_\eps \right] 
a_\eps^{p-1} + C\eps^2 a_\eps \eps^{-2} \Vert v(0) \Vert_{L^2}^{2p-2} \\
&\leq Ca_\eps^p + C\eps^2 T_\eps a_\eps^p.
\end{align*}
\end{proof}
With help of Lemma \ref{lemma:L2exp} we can finally prove stability in $L^2$.

\begin{theorem}
Consider for $m>4$ and $k\in(0,m-4)$ the exit time
\[
\tau_\eps = \inf \{ t \in [0,T_\eps] : \Vert v(t) \Vert_{H^{-1}} > c_0 \eps^m \text{ or }
 \Vert v(t) \Vert_{L^2} > C \eps^{k+1} \},
\]
where $T_\eps = \eps^{-N}$ for fixed large $N>0$.
Let also for some $\nu\in(0,1)$
\[
\Vert v(0) \Vert_{H^{-1}} \leq\nu \eps^m \quad \text{and} \quad
\Vert v(0) \Vert_{L^2} \leq \nu \eps^{k+1}
\]
and also assume for the noise strength that for some small $\tilde\kappa>0$
\[
\eta_0 \leq C \eps^{2m+1+\tilde\kappa} \quad \text{and} \quad \eta_2 \leq C \eps^{2k+4+\tilde\kappa }.
\]
Then the probability $\prob ( \tau_\eps < T_\eps )$ is smaller than any power of $\eps$,
as $\eps$ tends to $0$.
\end{theorem}

\begin{proof}
We have
\[
\prob \left( \tau_\eps < T_\eps \right) 
\leq \prob \left( \Vert v(\tau_\eps) \Vert_{L^2} > \eps^{k+1} \right) 
+ \prob \left( \Vert v(\tau_\eps) \Vert_{H^{-1}} >  \eps^m \right)
\]
Now, using the $H^{-1}$ result of Theorem \ref{thm:H-1stabnew} we have for any $l>1$
\[ 
\prob \left( \Vert v(\tau_\eps) \Vert_{H^{-1}} >  \eps^m \right) \leq C_l\eps^l\;.
\]
Moreover, by Lemma \ref{lemma:L2exp}
with Chebychev's inequality
\begin{align*}
\prob \left( \Vert v(\tau_\eps) \Vert_{L^2} > C\eps^{k+1} \right) 
&\leq C \eps^{-2p(k+1)} \E \Vert v(\tau_\eps) \Vert^{2p}_{L^2}   \\
&\leq C \eps^{-2p(k+1)} \eps^2 \left[\eps^{-2} + T_\eps \right] a_\eps^p 
= C \left( \eps^{-2(k+1)}  a_\eps  \right)^p \left[1 + \eps^2 T_\eps \right] 
\\& = C \left( \eps^{-2(k+2)}  K_\eps  \right)^p \left[1 + \eps^2 T_\eps \right] .
\end{align*}
Now by our assumptions the bracket is bounded by $\eps^{\tilde\kappa}$ and thus 
choosing $p$ large enough yields the result.
\end{proof}

\section{Estimates}
\label{sec:est}

In this final section we give all the estimates that were needed throughout this work.
Compared to the deterministic counterpart we need to bound higher order derivatives.
We start with estimating with respect to the $H^{-1}_0$ -- norm to conclude the first
part of section 4.

\begin{lemma}
\label{lem:estH-1}
For $i = 1, 2$ let $\psi^\xi_i$ be the orthonormal basis from Theorem \ref{LinCH}
and $\tilde{u}^\xi$ the bubble as constructed in Theorem \ref{thm:slowmf}. Further 
subindices will denote partial derivates with respect to $\xi$. The following estimates
hold true
\begin{align*}
&\Vert \tilde{u}^\xi_j \Vert_{H^{-1}} =\CO(1), \quad \quad  \qquad
\Vert \psi^\xi_{i,j} \Vert_{H^{-1}} = \CO(\eps^{-1}) \\
&\Vert \tilde{u}^\xi_{ij} \Vert_{H^{-1}} = \CO( \eps^{-1/2}), \qquad
\Vert \psi^\xi_{i,jk} \Vert_{H^{-1}} = \CO( \eps^{-3/2}).
\end{align*}
\end{lemma}

\begin{proof}
In Section 3 of \cite{CHBoundary} it was proved that
\begin{equation}
\label{e:Muxi}
\left( \frac{\partial u^\xi}{\partial x_i}, \frac{\partial u^\xi}{\partial x_j} \right)
= C \rho^2 \delta_{ij} + \CO(\rho^3) + \CO(\eps\rho^{-1}) + \CO(\exp).
\end{equation}
Using the relation
\begin{align}
\label{eq:xder}
\tilde{u}^\xi_i = \frac{\partial u^\xi}{\partial x_i} + \CO(\exp),
\end{align}
$\Vert \tilde{u}^\xi_j \Vert =\CO(1)$ is established. Furthermore, the bound 
$\Vert \psi^\xi_{i,j} \Vert = \CO(\eps^{-1})$ is part of Theorem \ref{LinCH}.

By definition, for $g \in H^{-1}$ we can find $f_1, f_2 \in L^2$ such that
\[
g = \nabla \cdot f =
\frac{\partial f_1}{\partial x_1} + \frac{\partial f_2}{\partial x_2}
\]
and the norm on $H^{-1}$ is given by
\[
\Vert g \Vert^2 = \inf_{g = \nabla \cdot f} \int_\Omega \vert f \vert^2 \, dx.
\]
Therefore, with \eqref{eq:xder} and choosing $f_j = \frac{\partial u^\xi}{\partial x_i}$, 
we have
\[
 \Vert \tilde{u}_{ij}^\xi \Vert^2 \leq \Vert \partial_{x_i} u^\xi \Vert_{L^2}^2 + \CO(\exp) = \CO( \eps^{-1/2}),
\]
where the $L^2$ estimate will be established in Lemma 5.2. 
The same argument yields 
$\Vert \psi^\xi_{i,jk} \Vert \leq \Vert \psi^\xi_{i,j} \Vert_{L^2}.$

In light of Theorem \ref{LinCH} (iii) we compute
\begin{align*}
\left\Vert \partial_j \frac{\tilde{u}^\xi_k}{\Vert \tilde{u}^\xi_k \Vert} \right\Vert_{L^2} &=
\left\Vert - \frac{\tilde{u}^\xi_k \langle \tilde{u}^\xi_k, \tilde{u}^\xi_{kj} \rangle}{\Vert \tilde{u}^\xi_k \Vert^3} + \frac{\tilde{u}^\xi_{kj}}{\Vert \tilde{u}^\xi_k \Vert} \right\Vert_{L^2}\\
&\leq \frac{\Vert u^\xi_k \Vert_{L^2} \Vert u^\xi_{kj} \Vert}{\Vert u^\xi_k \Vert^2}
+ \frac{\Vert u^\xi_{kj} \Vert_{L^2}}{\Vert u^\xi_k \Vert} \\
&\leq \frac{\Vert u^\xi_k \Vert_{L^2} + \Vert u^\xi_{kj} \Vert_{L^2}}{\Vert u^\xi_k \Vert}
\end{align*}
With the already proven bound $\Vert u^\xi_k \Vert = \CO(1)$,
$\Vert u^\xi_k \Vert_{L^2} = \CO(\eps^{-1/2})$ and $\Vert u^\xi_{kj} \Vert_{L^2} = \CO(\eps^{-3/2})$
(cf. Lemma 5.2) we obtain
\[
\left\Vert \partial_j \frac{\tilde{u}^\xi_k}{\Vert \tilde{u}^\xi_k \Vert} \right\Vert_{L^2}
= \CO( \eps^{-3/2}).
\]
Finally, by the definition in Theorem \ref{LinCH} we derive
\begin{align}
\label{psiL2}
\Vert \psi_{i,j}^\xi \Vert_{L^2} &\leq \sum_k \vert \partial_j a^\xi_{ki} \vert \frac{\Vert \tilde{u}^\xi_k \Vert_{L^2}}{\Vert \tilde{u}^\xi_k \Vert} + \CO(\eps^{-3/2})
\nonumber \\
&\leq C \eps^{-1/2} \sum_k  \vert \partial_j a^\xi_{ki} \vert + \CO(\eps^{-3/2}) = \CO(\eps^{-3/2}),
\end{align}
where we used that the matrix $(a_{ki}^\xi)$ does depend smoothly on $\xi$ and is non-singular.
\end{proof}

We conclude with the estimates with respect to $L^2$ which were needed for section 4.3.

\begin{lemma}
Under the same assumptions as in Lemma \ref{lem:estH-1} the following estimates hold true
\begin{align*}
&\Vert \tilde{u}^\xi_j \Vert_{L^2} =\CO(\eps^{-1/2}), \quad \quad \,
\Vert \psi^\xi_{i} \Vert_{L^2} = \CO(\eps^{-1/2}) \\
&\Vert \tilde{u}^\xi_{ij} \Vert_{L^2} = \CO( \eps^{-3/2}), \quad \quad
\Vert \psi^\xi_{i,j} \Vert_{L^2} = \CO( \eps^{-3/2}).
\end{align*}
\end{lemma}

\begin{proof}
First, we observe that by Theorem \ref{thm:slowmf} it suffices to analyse the partial
derivatives of $u^\xi$ as the correction term $v^\xi$ and all its derivatives are 
exponentially small.

By Lemma \ref{thm:stat} and \ref{eq:mass} we have
\begin{align}
\label{eq:L2est2}
\frac{\partial u^\xi}{\partial \xi_i} &= \eps^{-1} \frac{\partial U^\star}{\partial r}
\frac{\partial r}{\partial \xi_i} + \eps^{-2} \frac{\partial U^\star}{\partial \rho}
\frac{\partial a^\xi}{\partial \xi_i} \nonumber \\
&= \left[ \eps^{-1} U^\prime \left( \frac{r-\rho}{\eps} \right) + \CO(1) \right] \frac{\partial r}{\partial \xi_i},
\end{align}
where we defined $r = \vert x - \xi \vert$. We use the radial geometry of the problem and the fact that $U^\prime$ localizes around the boundary of the bubble. For some small 
$\delta > 0$ we consider the ring 
$\Omega_\delta = 
\left\{ x : \left| \vert x - \xi \vert - \rho \right| \leq \delta \right\}.$

We compute
\begin{align*}
\eps^{-2} \int_{\Omega_\delta} U^\prime \left( \frac{r-\rho}{\eps} \right)^2 
\left( \frac{\partial r}{\partial \xi_i}  \right)^2 \, dx
&\leq \eps^{-2} \int_{\Omega_\delta} U^\prime \left( \frac{r-\rho}{\eps} \right)^2 \, dx
\\
&\leq C \eps^{-1} \int_{\vert \eta \vert \leq \delta / \eps} U^\prime(\eta)^2 
(\eps \eta + \rho) \, d\eta \\
&\leq C\rho \eps^{-1} \int_{-\infty}^\infty U^\prime(\eta)^2 \, d\eta 
\leq C \eps^{-1}.
\end{align*}
On the set $\Omega \setminus \Omega_\delta$ we utilize 
$\vert U^\prime(\eta) \vert \leq ce^{-c\vert \eta \vert}$ and derive
\begin{align*}
\eps^{-2} \int_{\Omega \setminus \Omega_\delta} U^\prime \left( \frac{r-\rho}{\eps} \right)^2 
\left( \frac{\partial r}{\partial \xi_i}  \right)^2 \, dx
\leq C\eps^{-2} e^{-c \delta / \eps}\, \vert \Omega \setminus \Omega_\delta \vert
= \CO( \exp).
\end{align*}
Combined with \eqref{eq:L2est2} this shows 
$\Vert \tilde{u}^\xi_j \Vert_{L^2} = \CO(\eps^{-1/2}).$ Estimating the second order
derivatives can be carried out analogously.

Definition \ref{def:psi}, Lemma \ref{lem:estH-1} and the $L^2$--estimate of 
$\tilde{u}^\xi_j$ directly yield $\Vert \psi^\xi_{i} \Vert_{L^2} = \CO(\eps^{-1})$.
The bound for the second derivatives was established in \eqref{psiL2}.

\end{proof}

\newpage

\nocite{*}
\bibliography{bibliography}
\bibliographystyle{siam}

\end{document}